\documentclass{article}

\usepackage{PRIMEarxiv}

\usepackage[utf8]{inputenc} 
\usepackage[T1]{fontenc}    
\usepackage{hyperref}       
\usepackage{url}            
\usepackage{booktabs}       
\usepackage{amsfonts}       
\usepackage{nicefrac}       
\usepackage{microtype}      
\usepackage{graphicx}
\graphicspath{{media/}}     
\usepackage{libertine}
\usepackage{libertinust1math}
\usepackage{dsfont}  
\usepackage{nicefrac}
\usepackage{todonotes}
\usepackage{algorithm}
\usepackage{algorithmicx}

\usepackage[T1]{fontenc}
\usepackage{booktabs}
\usepackage{subcaption}
\usepackage{amsmath}
\usepackage{amssymb}
\usepackage{mathtools}
\usepackage{amsthm}
\usepackage{ifthen}
\usepackage{graphicx}
\usepackage{subcaption}
\usepackage{algorithm}%
\usepackage{algorithmicx}%
\usepackage{algpseudocode}%
\usepackage{listings}%

\theoremstyle{plain}
\newtheorem{theorem}{Theorem}
\newtheorem{lemma}[theorem]{Lemma}
\newtheorem{corollary}[theorem]{Corollary}
\newtheorem{proposition}[theorem]{Proposition}

\newtheorem{definition}[theorem]{Definition}

\newcommand{\fa}[2]{\forall #1 \ifthenelse{\equal{#2}{>0}}{>0}{\in #2} \;}
\newcommand{\ex}[2]{\exists #1 \ifthenelse{\equal{#2}{>0}}{>0}{\in #2} \;}
\renewcommand{\P}{{\mathbb{P}}}

\newcommand{\minimize}{\mathop{\mathrm{minimize}}}

\newcommand{\C}{{\bf (C)}}
\newcommand{\E}{{\mathbb{E}}}
\newcommand{\R}{{\mathbb{R}}}
\newcommand{\N}{{\mathbb{N}}}

\newcommand{\X}{{\mathcal{X}}}
\newcommand{\F}{{\mathcal{F}}}
\newcommand{\B}{{\mathcal{B}}}
\newcommand{\Y}{{\mathcal{Y}}}
\newcommand{\sgn}{\mathrm{sgn}}
\newcommand{\G}{{\mathcal{G}}}
\newcommand{\D}{{\mathcal{D}}}

\newcommand{\cE}{{\mathcal{E}}}
\newcommand{\cC}{{\mathcal{C}}}

\newcommand{\eps}{\varepsilon}
\newcommand{\abs}[1]{{\left|#1\right|}}
\newcommand{\norm}[1]{{\left\|#1\right\|}}
\newcommand{\set}[3]{{\left\{ #1 \in #2\ \middle|\ #3 \right\}}}
\newcommand{\ie}[3]{{\ifthenelse{\equal{#1}{}}{#2}{#3}}}
\newcommand{\der}[3]{{
\ifx&#2&#3&
    \nabla #1%
\else
    \ifx&#3&
        \nabla #1 \left( #2 \right)%
    \else 
        \left\langle \nabla #1 \left( #2 \right), #3 \right\rangle
    \fi
\fi
}}

\newcommand{\cder}[3]{{
    \partial_C#1\ie{#2#3}{}{\left(
            \ie{#2}{\cdot\;}{#2};
            \ie{#3}{\;\cdot}{#3}
    \right)}
}}
\newcommand{\cdder}[3]{{
     {#1}^{\circ}
     \ie{#2#3}{}{(
        \ie{#2}{\cdot\;}{#2};
        \ie{#3}{\;\cdot}{#3}
     )}
}}

\newcommand{\Xopt}{\bar{{\mathcal X}}}
\newcommand{\xo}{\bar{x}}
\newcommand{\Mirror}{\mathcal M}
\newcommand{\braket}[2]{{\left\langle#1, #2\right\rangle}}
\newcommand{\mul}[2]{{\left\langle#1, #2\right\rangle}}
\newcommand{\Up}{{\bf (U_p)}}
\newcommand{\V}{{\mathcal V}}

\title{Exact Penalty Method for Variationally Coherent  Stochastic Programming Problems
}

\author{
  Bogdan K. Jastrzębski \\
  Faculty of Mathematics and Information Science \\
  Warsaw University of Technology \\
  Warsaw \\
  \texttt{bogdan.jastrzebski.dokt@pw.edu.pl} \\
  \AND
  Radosław Pytlak \\
  Faculty of Mathematics and Information Science \\
  Warsaw University of Technology \\
  Warsaw \\
  \texttt{radoslaw.pytlak@pw.edu.pl} \\
}

\begin{document}
\maketitle

\begin{abstract}
The paper concerns optimization problems with general equality and inequality constraints and with constraints expressed by a convex set. In order to solve these problems, the general constraints are treated by an exact penalty functions while the others by mirror descent approach. The paper introduces a constraint qualification condition under which the solution of the optimization problem with an exact penalty function and constraints defined by the convex set is a solution of the original problem with constraints. The paper extends results on exact penalty functions to the case when together with general equality and inequality constraints additional constraints defined by a convex set are present. In order to solve the optimization problems with exact penalty functions, a mirror descent algorithm is proposed. It is assumed that instead of using gradients of functions defining constrained optimization problems, their stochastic approximations can be applied. The paper establishes global convergence of the proposed method under the assumption that applied exact penalty functions lead to variationally coherent optimization problems. Since exact penalty functions are not differentiable, the concept of variationally coherent problems is extended to the problems defined by functions exhibiting Clarke's generalized gradients. The behavior of the proposed method is illustrated by some numerical examples.  


\end{abstract}


\section{Introduction}


In various applications, especially in large scale optimization problems in statistics, first-order methods with stochastic oracle rank among the most popular. Very often, these problems include various equality and inequality constraints, are non-differentiable and non-convex. In this work, we study the convergence of the stochastic mirror descent (SMD) algorithm with exact penalization (EP) for solving variationally coherent problems admitting generalized derivative in Clarke's sense. We further study the algorithm numerically in solving online-learning (OL) problems with constraints.

Non-convex optimization problems with stochastic first-order oracles appear frequently in various areas of mathematics. In statistics, calculating estimators on very large datasets often involves online-learning (OL), where only a random fraction of the dataset is used at each step of the optimization procedure. As the subset is random, the algorithm has access only to stochastic gradient.

When solving optimization problem with constraints, one can use the exact penalty approach (\cite{DiPillo}, \cite{Han}, \cite{Piet}). The constrained problem is translated to a problem without constraints, but with an additional penalty term, such that the solution to the new problem is in the feasible set. Often the exact penalty makes the objective function non-differentiable, but differentiable in some broader sense, e.g., the function can admit a subgradient or Clarke's generalized derivative.

We introduce a generalized notion of variational coherence and extend the proof presented in \cite{boyd} to non-differentiable settings, but differentiable in the Clarke's sense. We study the algorithm numerically, applying the methodology to solve stochastic problems with equality constraints via exact penalty method. The analysis of the convergence of the algorithm is motivated by the need to assert convergence, improve the algorithm's speed and better understand the factors upon which convergence depends.

The paper provides the theoretical background for using SMD algorithm for solving optimization problems which not only exhibit simple constraints on decision variables $x\in {\cal X}$ but also general equality and inequality constraints expressed by some functions of decision variables. The SMD algorithm is used to solve optimization problems which are results of applying an exact penalization procedure to these problems. The effects of the exact penalization are optimization problems with simple constraints, which are described by non--differentiable functions. It is shown show that the SMD algorithm presented in \cite{boyd} can be successfully used to solve these problems provided that  objective functions of these functions admit generalized directional derivatives in the Clarke's sense and the problems are variationally coherent according to the definition stated in the paper. In particular, if the functions defining the original problem with constraints are continuously differentiable, then the first requirement is satisfied. 

In order to justify the approach presented in the paper it is necessary to extend the results on exact penalization presented so far in the literature (for example in \cite{DiPillo},\cite{Han},\cite{Piet}) since optimization problems we consider include not only general equality and inequality constraints (as in \cite{DiPillo},\cite{Han},\cite{Piet}) but also constraints of the type $x\in {\cal X}$. The extension requires using different constraint qualification to the one used so far in the context of exact penalization applied to mathematical programming problems.


We consider functions $f$ defined on $n$--dimensional vectors space $\V$ with norm $\|\cdot\|$. $\nabla f(x)$ means the gradient of $f$ with respect to $x$ which is treated as an element of the dual space $\Y=\V^\star$ with the norm $\|y\|_\star = \sup_{x\in \V} \left \{ \left \langle y,x\right \rangle |\ \|x\|\leq 1\right \}$, $\der{f}{x}{v}$ is its dual paring with vector $v \in \Y$, Furthermore, by $Df(x;d)$ we denote the directional derivative of $f$ at $x$ in the direction $d$,
likewise $\cdder{f}{x}{v}$ Clarke's generalized directional derivative  and generalized gradient $\cder{f}{x}{v}$. Functions $\R \to \R$ applied to $\R^n$ are elementwise, e.g. $\abs{x} = [\abs{x_0}, \dots, \abs{x_{n-1}}]^T$ and also $x \leq 0 \iff x_i \leq 0,\ i=1,\ldots,n$.






\section{Problem Setup}








Consider the nonlinear programming problem:
\begin{align}
\minimize \quad & f(x)    \label{consP1}  \\
\text{over all} \quad & x \in \X \nonumber \\ 
\text{subject to} \quad & h_i(x) = 0, \quad i \in E     \label{consP21} \\
& g_j(x) \leq 0, \quad j \in I.  \label{consP3}
\end{align}
Here, $\X$ is a convex compact subset of the space $\V$. We will call (\ref{consP1})--(\ref{consP3}) the optimization problem $\C$. For the simplicity of the presentation we assume that $E=\{1,2,\ldots,n_E\}$ and $I= \{n_E+1,\ldots,n_E+n_I\}$.

We assume that $f(x) = \E[F(x;\omega)]$ for some stochastic function $F:\X\times \Omega\rightarrow \R$ defined on underlying probability space $(\Omega,\F,\P)$. With respect to $F$ we impose the following assumptions:
\begin{description}
    \item[{\bf (A1)}] $F(x;\omega)$ is continuously differentiable in $x$ for almost all $\omega \in \Omega$.
    \item[{\bf (A2)}] There exists a finite $M$ such that $\E[\|\der{F}{x;\omega}{}\|^2_{\star}]\leq M^2$ for all $x\in \X$.
    \item[{\bf (A3)}] Functions $h_i$, $i\in E$, $g_j$, $j\in I$ are continuously differentiable.
\end{description}






We aim to solve the optimization problem $\C$ by applying an exact penalty function to the constraints. Instead of solving the problem $\C$, we would like to solve the problem
\begin{equation}
    \minimize \ %
        P_p(x) := f(x) + p \| g(x)_{+}, h(x)\|
    \quad \text{over all} \ x \in \X \label{unconsP1} 
\end{equation}
where $p$ is a nonnegative real number, $(g(x)_{+})_j = \max(0,g_j(x))$ and $\|\cdot \|$ is any norm in $\R^{n_E+n_I}$. Here, the penalty term $\left \| g(x)_{+},h(x) \right \|$ is non--differentiable, but differentiable in the Clarke's sense. Throughout the paper we will use the notation $M_{\infty}(x) = \left \| g(x)_{+},h(x) \right \|_{\infty}$. The problem (\ref{unconsP1}) is called $\Up$.

\subsection{Constraint Qualification}

The exact penalty method constructs a penalty function such that, for a finite penalty parameter $p$, the minimizer of the unconstrained problem $\Up$ coincides with the minimizer of the original constrained problem $\C$. This contrasts with classical penalty approaches, which require the penalty parameter to approach infinity to guarantee feasibility. However, this equivalence holds only under specific regularity assumptions, known as constraint qualifications (CQ). Without these conditions, the penalty function may introduce artificial local minima outside the feasible region, or the required parameter $p$ may diverge to infinity.

To guarantee that solving the problem $\Up$ strictly solves the problem $\C$, we must ensure the constraint landscape is well-behaved. Let us define the set of admissible directions from a specific point $x \in \X$:
\begin{equation}
    \D(x) \coloneq \set{d}{\V}{x + d \in \X} 
\end{equation} 
and the set of strict descent directions, representing the admissible directions in which all inequality constraint functions locally strictly decrease:
\begin{equation}
    \D_<(x) \coloneq \set{d}{\D(x)}{(\forall j \in I) \der{g_j}{x}{d} < 0}.
\end{equation}

With the notation clarified, the constraint qualification condition establishes the required geometric regularity and takes the following form:

\begin{definition}[Constraint Qualification -- {\bf (CQ)}]\label{def:cq} 
The triple $(\X, g, h)$ satisfies the constraint qualification condition {\bf (CQ)} iff for all $x \in \X$:
\begin{equation}
   \D_<(x) \neq \emptyset 
\end{equation}
and, if equality constraints are present ($E \neq \emptyset$), it also holds that:
\begin{equation}
    0 \in \rm{interior}\left [ \cE(x)\right ]
    \label{CQa}
\end{equation}
where
\begin{equation*}
    \cE(x) = \set{\der{h}{x}{d}}{\R^{n_E}}{d \in \D_<(x)}.
\end{equation*}
\end{definition}

The constraint qualification {\bf (CQ)} is structurally similar to those stated in \cite{Han} (Definition 2.1), \cite{DiPillo} (Assumption B), and in \cite{py99}. The following lemma (equivalent to Theorem 2.2 in \cite{Han}) establishes that under {\bf (CQ)}, the local minima of the original problem $\C$ and the unconstrained penalty problem $\Up$ coincide for a sufficiently large penalty parameter $p$.

Constraint qualification allows to uniformly bound the directional derivatives of the constraint functions. If we are at a point $x$ that is not perfectly feasible, {\bf (CQ)} guarantees the existence of a target point $v \in \X$ such that moving from $x$ towards $v$ strictly decreases the constraint violations.

In the following analysis, we will frequently refer to the directional derivatives of the individual constraint components. Let us define the directional operators for the equality and inequality constraints along a direction $d$:
\begin{align}
    \xi_i(x,d) &= \begin{cases}
        \langle \nabla h_i(x),d \rangle & \text{if } h_i(x) > 0 \\
        | \langle \nabla h_i(x),d \rangle | & \text{if } h_i(x) = 0 \\
        -\langle \nabla h_i(x),d \rangle & \text{if } h_i(x) < 0 
    \end{cases} \label{DirDer1} \\
    \zeta_j(x,d) &= \begin{cases}
        \langle \nabla g_j(x),d \rangle & \text{if } g_j(x) > 0 \\
        \max(0, \langle \nabla g_j(x),d \rangle) & \text{if } g_j(x) = 0 \\
        0 & \text{if } g_j(x) < 0 
    \end{cases} \label{DirDer2}
\end{align}

With these operators defined, we can state the local bounding property guaranteed by the constraint qualification.

\begin{lemma}\label{l2}
Assume ${\bf (CQ)}$. For any $\tilde{x}\in\X$, $\varepsilon > 0$,
there exists a neighborhood $\B(\tilde{x}, \varepsilon)$ of $\tilde{x}$, $K_1, K_2 > 0$
with the following properties: given any 
$x\in \X$ such that $x\in\B(\tilde{x},\varepsilon)$ there exist $v,\hat{v}\in \X$ such that
\begin{align}
& \sgn(h_i(x)) \left \langle \nabla h_i(x), v - x\right \rangle = -K_1 \frac{|h_i(x)|}{M_\infty(x)}, \quad \forall i\in E, \label{CQ1} \\
& \left \langle \nabla g_j(x), v - x \right \rangle \leq -K_2,\ \forall j\in I. \label{CQ2}
\end{align}
\end{lemma}


\begin{proof} 
The proof given in the Appendix \ref{appendix:proofl2} follows the lines of Lemma 3.6.1 in \cite{py99}.
\end{proof}

The directional derivative of the full penalty function, $DP^\beta_p(x;d)$, decomposes into the objective gradient and a scaled penalty directional term: 
\begin{equation}
    DP^\beta_p(x;d) = \langle \nabla f(x),d \rangle + p \Delta^\beta(x,d).
\end{equation}
Depending on the chosen norm, for an infeasible point $x$, the penalty derivative component $\Delta^\beta(x,d)$ takes the following forms:
\begin{eqnarray}
&{\displaystyle 
    \Delta^1(x,d) = \sum_{i\in E}\xi_i(x,d) + \sum_{j\in I} \zeta_j(x,d),} \label{dp1} \\
&{\displaystyle 
    \Delta^\beta(x,d) = \norm{g_{+}(x),h(x)}^{1-\beta}_\beta 
    \cdot \bigg[ \sum_{i\in E} | h_i(x) |^{\beta -1}\xi_i(x,d) + \sum_{j\in I_+(x)} (g_j(x))^{\beta - 1} \zeta_j(x,d) \bigg], }\label{dpbeta} \\
&{\displaystyle 
    \Delta^\infty(x,d) = \max \Big[ \max_{i\in E(x)}\xi_i(x,d), \max_{j\in I(x)} \zeta_j(x,d) \Big],} \label{dpinf}
\end{eqnarray}
where the active maximum index sets are defined as $E(x) = \{i\in E \mid |h_i(x)| = M_\infty(x)\}$ and $I(x) = \{j\in I \mid (g_{j}(x))_{+} = M_\infty(x)\}$.

\begin{theorem}\label{t1p}
Suppose that the assumptions {\bf (A1)}--{\bf (A3)} are satisfied on a neighborhood of the point $\bar{x}$, which is a strict local minimum of the problem {\bf (C)} (meaning $\bar{x}$ is feasible with respect to all constraints, and $\bar{x}\in \X$). Assume further that at the point $\bar{x}$ the constraint qualification {\bf (CQ)} holds. Then for each norm $\|\cdot\|$ in $\V$, there exists a $\bar{p} > 0$ such that for all $p\geq \bar{p}$, the point $\bar{x}$ is a local minimum of $P_p(x)$ on the set $\X$. 
\end{theorem}

First, the objective gradient is bounded:
\begin{equation}
    \langle \nabla f(x(p)), v - x(p) \rangle \le C K_3
\end{equation}
By the constraint qualification:
\begin{equation}
    {\rm sgn} (h_i(x(p))) \langle \nabla h_i(x(p)), v - x(p) \rangle \le -K_1 \frac{|h_i(x(p))|}{M_\infty(x(p))}
\end{equation}
Using the $\xi_i$ operator, this is equivalent to:
\begin{equation}
    \xi_i(x(p), v - x(p)) \le -K_1 \frac{|h_i(x(p))|}{M_\infty(x(p))}
\end{equation}
Summing over the equality constraints:
\begin{eqnarray}
&{\displaystyle 
     \sum_{i\in E} |h_i(x(p))|^{\beta-1} \xi_i(x(p), v-x(p))
  \le \sum_{i\in E} |h_i(x(p))|^{\beta-1} \left( -K_1 \frac{|h_i(x(p))|}{M_\infty(x(p))} \right) }\nonumber \\
&{\displaystyle 
 = - \frac{K_1}{M_\infty(x(p))} \sum_{i\in E} |h_i(x(p))|^\beta = - K_1 \frac{\|h(x(p))\|_\beta^\beta}{M_\infty(x(p))}}\nonumber
\end{eqnarray}
Plugging the sum back into the full directional derivative:
\begin{eqnarray}
&{\displaystyle 
    DP^\beta_p(x(p); v-x(p))
  \le \langle \nabla f(x(p)), v-x(p)\rangle  + p \| h(x(p)) \|^{1-\beta}_\beta \left( -K_1 \frac{\|h(x(p))\|_\beta^\beta}{M_\infty(x(p))} \right) }\nonumber \\
&{\displaystyle 
 = \langle \nabla f(x(p)), v-x(p)\rangle - p K_1 \frac{\|h(x(p))\|_\beta}{M_\infty(x(p))} \le C K_3 - p c_\beta K_1 < 0}\nonumber
\end{eqnarray}
The last inequality follows from norm equivalence, as in finite spaces there exists $c_\beta > 0$ such that:
\[
    c_\beta M_\infty(x(p)) = c_\beta \norm{h(x(p))}_\infty \leq \norm{h(x(p))}_\beta
\]
Thus, the final formula is strictly negative for sufficiently large $p$.

Using the analytical mechanics established in the proof of Theorem \ref{t1p}, we can now state the converse theorem, assuring that our algorithmic target aligns perfectly with the original problem (cf. Theorem 4.1 in \cite{DiPillo}).

\begin{theorem}\label{t2p}
Assume that the assumptions of Theorem \ref{t1p} are satisfied. Then, there exists a threshold $\bar{p} > 0$ such that for any $p \geq \bar{p}$, if $x(p)$ is a local minimum point of the unconstrained problem $\Up$, then $x(p)$ is also a local minimum point of the constrained problem $\C$. 
\end{theorem}

\begin{proof}
Drawing on the arguments applied in Theorem \ref{t1p} (and mirroring Proposition 3.3 in \cite{DiPillo}), we first establish that $x(p)$ must be feasible. By Lemma \ref{l2}, for any infeasible point, there exists a direction that strictly decreases the penalty term. For $p \geq \bar{p}$, this penalty descent dominates the bounded objective gradient, meaning no infeasible point can be a local minimum. Thus, if $x(p)$ is a local solution to $\Up$, it must be strictly feasible with respect to the constraints of problem $\C$.

Let $\G = \{x \in \X \mid h_i(x) = 0, \forall i \in E; \ g_j(x) \leq 0, \forall j \in I\}$ denote the feasible set of the original problem. Since $x(p)$ is a local minimum of $\Up$, there exists a neighborhood $\B(x(p), \varepsilon)$ with $\varepsilon > 0$ such that:
\begin{equation}
    P_p(x(p)) \leq P_p(x), \quad \forall x \in \B(x(p), \varepsilon) \cap \X.
\end{equation}
Because $x(p)$ is feasible, $x(p) \in \G$, which means the penalty term at $x(p)$ evaluates to zero, giving $f(x(p)) = P_p(x(p))$. Furthermore, for any point $x$ that is also within the feasible set $\G$, $P_p(x) = f(x)$. Therefore, restricting the neighborhood to the feasible set yields:
\begin{equation}
    f(x(p)) = P_p(x(p)) \leq P_p(x) = f(x), \quad \forall x \in \B(x(p), \varepsilon) \cap \X \cap \G. \label{ep4a}
\end{equation}
Equation \eqref{ep4a} is the exact mathematical definition of a local minimum for problem $\C$, completing the proof.
\end{proof} 

The equivalence theorems above provide the foundational guarantee that by targeting the local minimizers of $\Up$, we are fundamentally solving $\C$. Consequently, we must design an algorithm capable of finding these minimizers. However, necessary optimality conditions typically require gradient evaluations. Because the exact penalty function $P_p(x)$ is inherently non-differentiable (due to the presence of the norm in its formulation), standard smooth optimization tools are insufficient. To formally evaluate descent directions and establish convergence on the set $\X$, we require a robust framework for non-smooth calculus. In the following section, we introduce the apparatus of Clarke's generalized gradients \cite{clarke2013functional} to bridge this theoretical gap.

Using parts of the proof of {\it Theorem \ref{t1p}} we are able to prove the following theorem (cf. Theorem 4.1 in \cite{DiPillo}).
\begin{theorem}\label{t3p}
Assume that the assumptions of Theorem \ref{t1p} are satisfied. Then there exists $\bar{p} > 0$ such that for $p\geq \bar{p}$ if $x(p)$ is a local minimum point of the problem ${\bf (U_p)}$ then $x(p)$ is also a local minimum point of the problem ${\bf (C)}$. 
\end{theorem}
\begin{proof}
Using arguments similar to those applied in the proof of {\it Theorem \ref{t1p}}, and also those which are presented in the proof of Proposition 3.3 stated in \cite{DiPillo}, we can show that there exists $\bar{p} > 0$ such that for any $p \geq \bar{p}$, if $x(p)$ is a local solution to the problem $\Up$ then $x(p)$ is feasible with respect to the constraints of the problem ${\bf (C)}$. It means that there exists $\B(x(p),\varepsilon)$ ($\varepsilon > 0$) such that
\begin{equation*}
    f(x(p)) = P_p(x(p)) \leq P_p(x),
    \quad
    \forall x\in \B(x(p),\ \varepsilon)\cap \X,\ p \geq \bar{p},
\end{equation*}
which implies that
\begin{equation}
    f(x(p)) = P_p(x(p)) \leq  P_p(x) = f(x),
    \quad
    \forall x\in \B (x(p), \varepsilon)\cap \X \cap \G,
    \label{ep4}
\end{equation}
$p\geq \bar{p}$, where $\G = \{x\in {\mathcal X} | \ h_i(x) = 0,\ i\in E,\ g_j(x) \leq 0,\ j\in I\}$. But (\ref{ep4}) states that $x(p)$ is a strict local minimum for the problem ${\bf (C)}$.


We assume that $\bar{x}$ is a local solution for the problem $\min_{x\in \X} P_{\bar{p}}(x)$ which means that there exists some neighbourhood $\B (\bar{x},\varepsilon)$ ($\varepsilon > 0$) such that 
\begin{equation*}
    f(\bar{x}) = P_{\bar{p}}(\bar{x}) \leq P_{\bar{p}}(x),
    \quad
    \forall x \in \B(\bar{x},\varepsilon) \cap \X
\end{equation*}
However, because $\bar{x}\in {\mathcal G}$, we also have 
\begin{equation*}
    f(\bar{x}) = P_{\bar{p}}(\bar{x}) \leq P_{\bar{p}}(x) = f(x),
    \quad
    \forall x\in \B (\bar{x},\varepsilon)\cap {\mathcal X}\cap {\mathcal G}.
\end{equation*}
which means that $\bar{x}$ is a local solution to the problem ${\bf (C)}$. 
\end{proof} 


The above theorems refer to local minimizers of problems $\C$ and $\Up$. As we will see in the next sections these results are needed to show that the proposed algorithm applied to the problems $\Up$ will find a minimizer of the problem $\C$.

\subsection{Clarke's Generalized Gradient}

The following definition of Clarke's generalized derivative will be used in the convergence analysis of the proposed algorithm.

\begin{definition}[Generalized Directional Derivative]
Let $\V$ be a vector space and $f: \V \to \R$ be a function. The generalized directional derivative $\cdder{f}{x}{v}$ of $f$ at $x$ in the direction $v$ is defined as:
\begin{equation}
\cdder{f}{x}{v} \coloneq \limsup_{
    \substack{
    y\rightarrow x \\
    t\downarrow 0
}
} \frac{f(y+tv)-f(y)}{t}
\end{equation}
\end{definition}

The basic properties of the Clarke's generalized directional derivative $\cdder{f}{x}{}$ are: it is finite, positively homogeneous: $\cdder{f}{x}{\lambda v} = \lambda \cdder{f}{x}{v}$, for $\lambda \in \R_+$, subadditive: $\cdder{f}{x}{v + w} \leq \cdder{f}{x}{v} + \cdder{f}{x}{w}$, upper semicontinuous: $ \limsup_{(u,z) \to (v,x)}\cdder{f}{z}{u} \leq \cdder{f}{x}{v}$, $K$--Lipschitz: $\|\cdder{f}{x}{v} - \cdder{f}{x}{u}\| \leq K\|v - u\|$ and $\cdder{f}{x}{-v} = \cdder{(-f)}{x}{v}$, for all $u,v,w,x,z$.
%
%

Note that due to positive homogeneity and subadditivity, the generalized directional derivative is in particular convex.

\begin{definition}[Clarke's Generalized Gradient] Let $\V$ be a vector space and $f: \V \to \R$ be a function. The generalized derivative $\cder{f}{}{} : \V \to \V^*$ in Clarke's sense is defined as (\cite{clarke2013functional}):
\begin{equation}
\begin{split}
    \partial_C f(x)
    &
    \coloneq \set{g}{\V^*}{\cdder{f}{x}{v} \geq \braket{g}{v}\ \fa{v}{\V}}.
\end{split}
\end{equation}
An element of the subdifferential $g\in \partial_Cf(x)$ is called the subgradient of $f$, and when it is needed, it will also be denoted as $g(x)$.
\end{definition}

\subsection{Variational coherence}
The aim of the paper is to show that an SMD algorithm can be used to solve problems $\Up$. In that way the method for solving constrained optimization problems by using exact penalty functions is proposed. This means that the constrained optimization problem $\C$ can be tackled by using penalty functions with much lower values of penalty parameters than in the case of standard penalty functions which, in order to guarantee feasible points, require these values approaching infinity (\cite{nocedal}). In this paper mechanisms (the updates of penalty parameters when solving the problem $\Up$), which should provide as small as possible values to guarantee feasiblity of local solutions to $\Up$, are not analyzed. Such schemes are considered, for example, in \cite{py99}, or \cite{toint}. 

The SMD scheme is a suitable candidate for solving problems $\Up$ since its variants comprise one of the
most widely used families of first--order methods in stochastic convex and non--convex optimization. In particular, a SMD variant proposed in \cite{boyd} exhibits globally convergent properties when applied to variationally coherent functions which represent broader class of functions than a class of quasi--convex functions.

Following \cite{boyd} the paper introduces a generalized definition of variational coherence for functions admitting Clarke's generalized derivative. Let $\Xopt = \arg\min_{x\in \X} f(x)$. Since $\X$ is compact and $f$ is continuous the set $\Xopt$ is closed and non--empty.

\begin{definition}[Variational Coherence]\label{definition-variational-coherence}
We say that the problem $\min_{x\in \X} f$ is variationally coherent if
for all $x \in \X$, and for all $\xo \in \Xopt$
\begin{align}
    & \cdder{f}{x}{\xo - x} \leq 0
    \label{varcoh-1-<=0}
    \\ & 
     \cdder{f}{x}{\xo-x} = 0  \implies x \in \Xopt. \label{varcoh-2-opt}
\end{align}
Notice that according to the definition of $\partial_Cf(x)$ the first condition (\ref{varcoh-1-<=0}) can be rephrased as $\braket{g}{\xo - x} \leq 0\ \forall g \in \partial_Cf(x)$.
\end{definition}

%

\subsection{The mirror descent algorithm}
The paper presents a first--order algorithm for solving optimization problems $\Up$, based on the SMD scheme proposed by Nemirovski and Yudin in \cite{nem} and studying further in various papers including  \cite{beck2002}, 
\cite{lan2012}, \cite{nesterov2009} and \cite{boyd}.

The SMD scheme is an iterative algorithm which at the $k$--th iteration makes the following update of a random variable $X_k$ which is an approximate to the optimization problem solution:
\begin{eqnarray}
    Y_{k+1} & = & Y_k - \gamma_k G_k \label{SMD_step1}\\
    X_{k+1} & = & \Mirror_h(Y_{k+1}). \label{SMD_step2}
\end{eqnarray}
Here, $G_k = g(X_k;\omega_k)$ is an i.i.d sample of the subgradient of $P_p$ evaluated in the dual space $\Y =\V^\star$ (from the distribution of the random variable $G$), and $\Mirror_h:\Y\rightarrow \X$ is a mirror map that translates the aggregated gradient from the dual space to the set of decision variables. $\gamma_k > 0$ are stepsizes which will be specified later.

The mirror map is defined in general by
\begin{eqnarray}
    &{\displaystyle \Mirror_h(y) = \arg\max_{x\in \X} \{ \langle y,x\rangle - h(x) \},} \label{GeneralMirror}
\end{eqnarray}
where $h:\X\rightarrow \R$ is the regularizer expressed by a continuous strongly convex function with the coefficient $K > 0$, i.e., the function satisfying for any $x,y\in \X$, $\lambda \in [0,1]$:
\begin{eqnarray}
    &{\displaystyle h(\lambda x + (1-\lambda)y) \leq \lambda h(x) + (1-\lambda)h(y) -\frac{1}{2}K \lambda (1-\lambda) \|x-y\|^2.} \label{MirrorReg}
\end{eqnarray}
In the convergence analysis which follows the crucial role is played by Fenchel coupling which is a measure of divergence in primal and dual spaces.


\begin{definition}[Fenchel Coupling]
Let $\V$ be a vector space and $h: \V \rightarrow \R$ be continuous and strongly convex.
The Fenchel coupling $F_h : \V \times \V^* \to \R$ is defined as:
\begin{equation}
F_h(x, y) \coloneq h(x) + h^*(y) - \braket{y}{x} \label{definition-fenchel-coupling},
\end{equation}
where $h^*:\V^*\to \R$ is the convex conjugate of $h$ defined as: $h^*(y) = \max_{x \in V} \{\braket{y}{x} - h(x)\}$.
\end{definition}
The basic properties of the Fenchel coupling, used in the convergence analysis, are 
\begin{eqnarray}
F_h(x, y_0) & \geq & \frac{K}{2} \|\Mirror_h(y_0) - x\|^2 \label{proposition-fc-properties-1}\label{Fenchel1}\\
F_h(x, y_1) & \leq & F_h(x, y_0)
+ \braket{y_1 - y_0}{\Mirror_h(y_0) - x}
+ \frac{1}{2K} \|y_0 - y_1\|^2_{*} \label{Fenchel2}
\end{eqnarray}
for all $x \in \V$ and $y_0,\,y_1 \in \V^*$.
In particular, the Fenchel coupling is non--negative. These properties are valid under the assumption that the Fenchel coupling is defined with the help of the $h$ regularizer which is $K$--strongly convex. The proof of these properties can be found in \cite{boyd}.

The inequality (\ref{proposition-fc-properties-1}) says that if $F_h(\xo, y_k)\rightarrow 0$ then $\Mirror_h(y_k)\rightarrow \xo$. However, in order to achieve convergence of SMD we need the additional assumptions:
\begin{description}
    \item[{\bf (A4)}] $\Mirror_h(y_k)\rightarrow \xo$ then $F_h(\xo, y_k)\rightarrow 0$,
    \item[{\bf (A5)}] the sequence of step sizes $\{\gamma_k\}$ satisfies the conditions:
    \begin{equation}
        \sum_{k=1}^{\infty}\gamma_{k}^{2}<\infty\ \ {\rm and}\ \  \sum_{k=1}^{\infty}\gamma_{k}=\infty. \label{assumption-4-RobbinsMonro}
    \end{equation}
\end{description}

\section{Convergence Analysis}
Our convergence analysis requires the notions of $\eps$--neighborhood of a set and the Fenchel $\delta$--zone of a set. Let assume that $\cC$ is a subset of $\X$ and $x\in X$ then the distance between $x$ and $\cC$ is defined as ${\rm dist}(\cC,x) = \inf_{y\in \cC} \|y-x\|$ and $\eps$--neighborhood of $\cC$ as $\B(\cC,\eps)= \{x\in \X \mid {\rm dist}(\cC,x)<\eps\}$. The Fenchel coupling between $y\in \Y$ and $\cC$ is defined by $F(\cC,y)=\inf_{x\in \cC}F(x,y)$ and the Fenchel $\delta$--zone of a set $\cC$ as $\B_F(\cC,\delta) = \{x\in \X \mid x=\Mirror_h(y)\ {\rm for\ some}\ y\in \Y\ {\rm with}\ F(\cC,y) <\delta\}$.

We now show an equivalent to Proposition 3.4 in \cite{boyd}, with the difference of using Clarke's generalized gradients instead of gradients. 


\begin{proposition}[Recurrence of $\eps$--neighborhoods and the Fenchel zones] 
\label{recurrence}
Suppose that $p$ is large enough so that local solutions to the problem $\Up$ are local solutions to the problem $\C$, and that $\Up$ is variationally coherent. If the assumptions ${\bf (A1)}$--${\bf (A4)}$ and the constraint qualification ${\bf (CQ)}$ hold, then  for any $\varepsilon > 0$ and $\delta > 0$ the sequence generated by the iterations (\ref{SMD_step1})--(\ref{SMD_step2}) enter $\B(\Xopt, \varepsilon)$ and $\B_F(\Xopt, \delta)$ infinitely many times, almost surely.
\end{proposition}

The proof in \cite{boyd} relies on continuity of the gradient. While the generalized gradient is not continuous, it is upper semi--continuous, which we show is sufficient for the proposition to hold true. Let us then prove the following lemma:
\begin{lemma}[Existence of bound $c$]
\label{lemmaBound}
For all $\varepsilon>0$ there exists $c_\varepsilon > 0$, such that for all $x \in \X \setminus \B(\Xopt, \varepsilon)$ and for all $\xo \in \Xopt$
\begin{equation}
\begin{split}
\cdder{f}{x}{\xo - x} \leq -c_\varepsilon.
\end{split} \label{lemma-thesis}
\end{equation}
\end{lemma}
\begin{proof}
Assume, contrarywise, there exists $\varepsilon > 0$, such that for all $c > 0$ exist $x_c \in \X \setminus \B(\Xopt, \varepsilon)$ and $\xo_c \in \Xopt$ which satisfy
\begin{equation}
\begin{split}
\cdder{f}{x_c}{\xo_c - x_c} > -c.
\end{split} \label{lemma-contra}
\end{equation}
Take any sequence $c : \N \to \R_+$, such that $\lim_{n \to \infty} c(n) = 0$ and define sequences:
\begin{equation}
    x(n) \coloneq x_{c(n)} \qquad \xo(n) \coloneq \xo_{c(n)}.
    \label{lemma-sequences}
\end{equation}
There exists then an increasing function $I : \N \to \N$, and elements $x_\infty \in \X\setminus \B(\Xopt, \varepsilon)$ and $\xo_\infty \in \Xopt$ such that subsequences $x_I \coloneq x \circ I$ and $\xo_I \coloneq \xo \circ I$ converge:
\begin{equation}
    \lim_{n \to \infty} (x_I(n), \xo_I(n)) = (x_\infty, \xo_\infty).
    \label{lemma-C}
\end{equation}
$\X\setminus \B(\Xopt, \varepsilon)$ is compact, as it is a difference of a compact set $\X$ and an open set $\B(\Xopt, \varepsilon)$, while $\Xopt$ is compact, as it is a closed subset of a compact set $\X$. The set $\X\setminus\B(\Xopt, \varepsilon) \times \Xopt$ is compact, being a product of compact sets. For these subsequences the following holds:
\begin{equation}
\lim_{n \to \infty}\cdder{f}{x_I(n)}{\xo_I(n) - x_I(n)} \geq 0 \label{lemma-lim-zero}
\end{equation}
as the sequence is bounded below by the sequence $-c_I \coloneq -c \circ I$, which by definition converges to zero:
\begin{equation}
    -c_I(n) < \cdder{f}{x_I(n)}{\xo_I(n) - x_I(n)},
    \quad \fa{n}{\N}
    \label{lemma-lower-bound}
\end{equation}
where the lower bound follows from assumption (\ref{lemma-contra}).
Eventually:
\begin{equation}
    \cdder{f}{x_\infty}{\xo_\infty - x_\infty} \geq 0 \label{lemma->=-zero}
\end{equation}
from:
\begin{align}
    0 \leq \lim_{n \to \infty}\cdder{f}{x_I(n)}{\xo_I(n) - x_I(n)} 
\leq \limsup_{
      (x, \xo) \to (x_\infty, \xo_\infty)
      }\cdder{f}{x}{\xo - x}\label{lemma-E-2} 
      \leq \cdder{f}{x_\infty}{\xo_\infty - x_\infty} 
\end{align}
where the first inequality follows from (\ref{lemma-lim-zero}), the second inequlity from limit properties, and most importantly the last inequality holds by the upper semi-continuity of $\cdder{f}{}{}$.

However, by assumption of variational coherence (\ref{varcoh-2-opt}), it also holds that:
\begin{equation}
    \cdder{f}{x_\infty}{\xo_\infty - x_\infty} < 0 \label{lemma-<-zero}
\end{equation}
given $x_\infty \notin \Xopt$. Statement (\ref{lemma-<-zero}) contradicts (\ref{lemma->=-zero}), which finally proves the proposition of the lemma.
\end{proof}

It is now possible to proceed with the proof of {\it Proposition \ref{recurrence}}, which is essentially the same as in \cite{boyd}. As in \cite{boyd} it consists of three steps. In the first step it is shown that the sequence $\{Y_k\}$ has martingal properties (by referring to the assumptions ${\bf (A1)}$--${\bf (A2)}$). In the second step the recurrence of $\eps$--neighborhoods is established with respect to $\B(\Xopt,\eps)$, while in the third step with respect to $\B_F(\Xopt,\delta)$ by referring to the assumption ${\bf (A4)}$. 

\begin{theorem}[Global convergence] 
\label{globalConvergence}
Suppose that the assumptions of {\rm Proposition \ref{recurrence}} hold together with ${\bf (A5)}$. Then $\{X_k\}$ converges with probability $1$ to a (possibly random) minimum point of $\C$.
\end{theorem}

The proof of {\it Theorem \ref{globalConvergence}} is different from the proof of Theorem 4.1 in \cite{boyd}) due to the fact that in the step (\ref{SMD_step1}) of the iterative process  subgradients instead of gradients are used.
\begin{proof}
Define a random sequence:
\begin{equation}
    F_k \coloneq F(\xo, Y_k)
\end{equation}
From {\it Lemma \ref{lemmaBound}}, for all $\varepsilon>0$, there exists $c_\varepsilon > 0$, such that for all $x \in \X \setminus \B(\xo, \varepsilon)$ and $\xo \in \Xopt$
\begin{equation}
\cdder{P}{x,p}{\xo - x} \leq -c_\varepsilon < 0
\end{equation} 
In terms of generalized derivative, for all $\varepsilon>0$ there exists $c_\varepsilon > 0$, such that for all $x \in \X \setminus \B(\Xopt, \varepsilon)$, $\xo \in \Xopt$ and $g \in \partial_C f(x)$
\begin{equation}
\braket{g}{\xo - x} \leq -c_\varepsilon < 0 \nonumber 
\end{equation} 
An upper bound to $F_{k+1} - F_k$ is derived as follows:
\begin{align}
F_{k+1} & = F(\xo, Y_{k+1})  = F(\xo, Y_{k} - \gamma_k G_k(X_k))
\label{recurrence-1}\\ 
& \leq F(\xo, Y_k) + \frac{1}{2K} \| Y_k - \gamma_k G_k(X_k) - Y_k\|^2_*
\nonumber \\
& + \braket{Y_k - \gamma_k G_k(X_k) - Y_k}{\Mirror_h(Y_k) - \xo}
\label{recurrence-2}\\ 
& = F_k + \frac{\gamma_k^2}{2K} \|G_k(X_k)\|^2_* - \gamma_k \braket{G_k(X_k)}{X_k - \xo}
\label{recurrence-3}\\ 
& = F_k + \frac{\gamma_k^2}{2K} \|G_k(X_k)\|^2_* + \gamma_k \braket{U_k - g(X_k)}{X_k - \xo} \label{recurrence-4}
\\ 
& \leq F_k + \frac{\gamma_{k}^2}{2K} \|G_k(X_k)\|^2_* + \gamma_{k}\braket{U_k}{X_k - \xo} - \gamma_k c_\varepsilon\label{recurrence-5}
\end{align}
where (\ref{recurrence-1}) follows from the algorithm definition, (\ref{recurrence-2}) from (\ref{Fenchel2}), (\ref{recurrence-3}) from terms rearranging, (\ref{recurrence-4}) from definition of $U_k$, and (\ref{recurrence-5}) from {\it Lemma \ref{lemmaBound}}.
\end{proof}
%



\begin{corollary}[Convergent subsequence] With probability 1 there exists a subsequence of $X_n$, that is convergent to a random minimum point.
\end{corollary}

As in \cite{boyd} (Theorem 5.3), under weaker assumption than variational coherence, local convergence can be established.
\begin{definition}[Weak Variational Coherence]\label{definition-weak-variational-coherence}
We say $f$ is weakly variationally coherent if 
there exists $\xo \in \Xopt$ such that
\begin{align}
    &
    \cdder{f}{x}{\xo - x} \leq 0,\ \forall x\in \X
    \\ &
    \cdder{f}{x}{\xo - x} = 0 \implies x \in \Xopt,
\end{align}
and for any $\xo \in \Xopt$ there exists $\varepsilon > 0$, such that for all $x \in \B(\xo, \varepsilon)$
\begin{align}
    \cdder{f}{x}{\xo - x} \leq 0
\end{align}
\end{definition}

\begin{theorem}[Convergence for weakly coherent problems]\label{theorem-wekly coherent}
Let $\Up$ be weakly variationally coherent with $p$ sufficiently large so that local minimizers of $\Up$ are local minimizers of $\C$. Suppose that the assumptions ${\bf (A1)}$--${\bf (A5)}$ and constraint qualification ${\bf (CQ)}$ are satisfied then $\{X_k\}$  converges with probability 1 to a (possibly random) minimum point of the problem $\C$.
\end{theorem}

\begin{figure}
    \centering
    \includegraphics[width=0.9\linewidth]{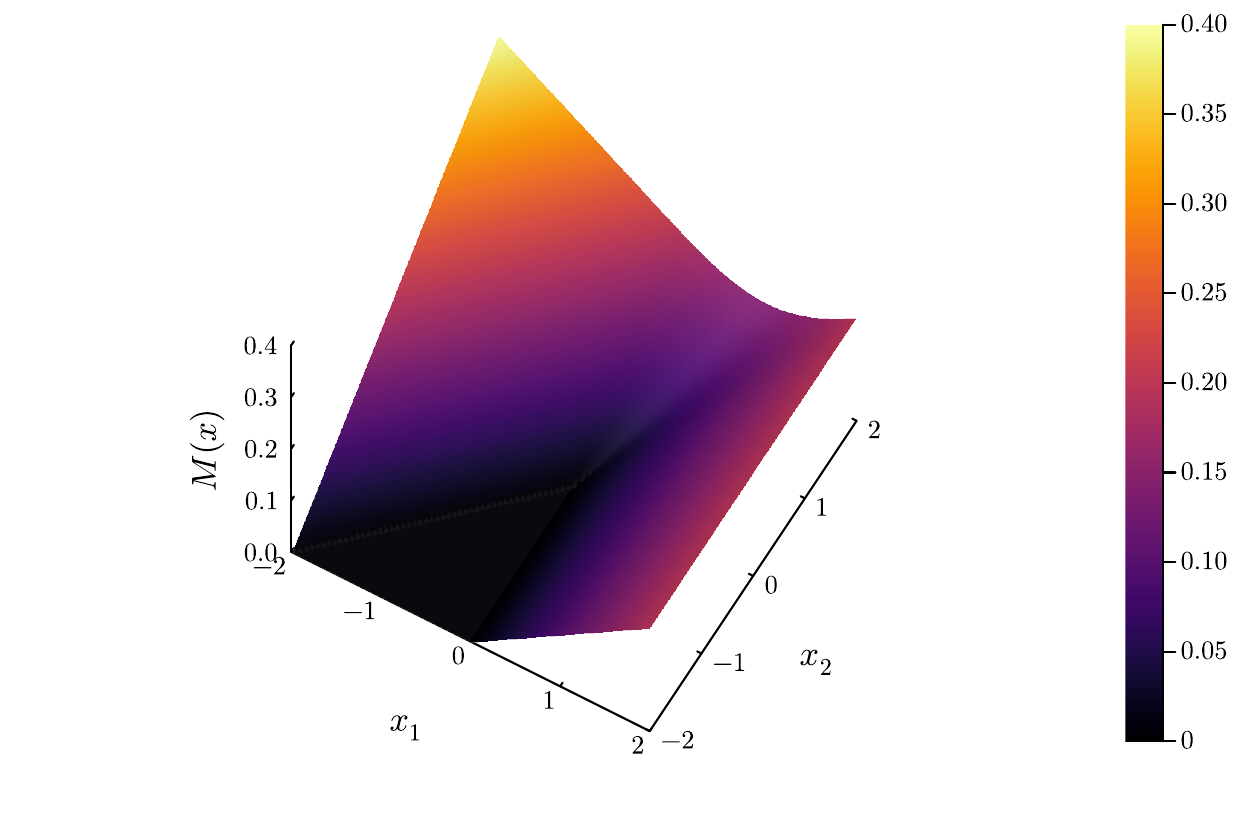}
    \caption{3D visualization of the well-behaved $l_2$ penalty function for constraints $x_1 \leq 0,\ x_2-x_1 \leq 0$. The distinctive feature of this formulation is that the function is locally differentiable outside the feasible set. At the same time it models well a sharp minimum defined by constraints.}
    \label{fig:3dsmooth}
\end{figure}

\section{The update rules for penalty parameters}

In general, there are two approaches to developing algorithms based on exact penalty functions. In the first approach, we assume that we have a limiting value of the parameter $p$ of the exact penalty function that guarantees that minimizing the function $P_p$ leads to the solution of the constrained problem. Algorithms based on this approach are essentially: 1) sophisticated descent methods for solving problems without constraints defined by nondifferentiable functions with a special structure (\cite{Conn73}); 2) or algorithms in which the direction of descent is determined by solving a linear or quadratic problem that is a local approximation of the problem $\C$ (\cite{Han77}, \cite{ConnPiet77}, \cite{Toint2011}).

In the second approach, the problem $P_p$ and the parameter $p$ adjustment are performed in parallel in such a way that a threshold value of the parameter $p$ is determined that ensures the solution of the problem $\C$. The second approach was initiated in the work \cite{Polak76}, where a testing function was introduced to ensure the proper selection of the parameter $p$ and the conditions that it must meet in order to determine the desired value of $p$ in the iterative process were formulated.

If simple constraints $\X$ are not present in problem $\C$, then the testing function refers to approximations of the Lagrange multipliers occurring in the KKT conditions of problem $\C$, obtained by solving linear or quadratic problems that are approximations of problem $\C$ (\cite{MayneMaratos79}, \cite{PantojaMayne79}). Basing the testing function on approximations of the Lagrange multipliers is possible due to the fact that the threshold value of the parameter $p$ guaranteeing the solution of the problem $\C$ can be set to the maximum value of the absolute values of the Lagrange multipliers occurring in the KKR conditions of the solution of the problem $\C$ (\cite{Luenberger70}].

If the formulation of problem $\C$ contains simple constraints $\X$, then the testing function is based on the optimal values of linear or quadratic problems that approximate problem $\C$ (\cite{MaynePolak80}, \cite{MaynePolak87}, \cite{pv98}).

However, the case of the optimization problem considered in this paper does not allow the application of any of the penalty parameter updating schemes described in the literature. Since the considered optimization problem contains simple constraints $\X$, we should apply a rule based on solving a quadratic problem approximating the problem $\C$; the use of stochastic gradient approximation for the objective function—--combined with the application of the proposed approach to optimization problems in machine learning (where the number of decision variables can run into the millions)—--precludes this.


First, we present a simplified framework where gradients of the objective function $f$ are evaluated exactly, and the strict domain constraint $\X$ is not present ($\X = \R^n$). For the theoretical analysis in this section, we assume that the sequence of iterates generated by the algorithm remains strictly bounded. Note that this assumption is a standard regularization and holds naturally under typical coercivity conditions on the objective function.

The update rule for the penalty parameter relies on the exact penalty function formulation:
\begin{equation}
    P^\beta_p (x) = f(x) + p \norm{g(x)_{+}, h(x)}_\beta, \quad 1 < \beta < \infty.
\end{equation}
For $\beta \in (1, \infty)$ the function $P^\beta_p$ is continuously differentiable at any strictly infeasible point. At these points, its gradient is given by $\der{P^\beta_p}{x}{} = \der{f}{x}{} + p \, g_\beta(x)$, where:
\begin{align}
    g_\beta(x) &= \sum_{i\in E} \sigma_i(x,\beta) \der{h_i}{x}{} + \sum_{j\in I} \eta_j(x,\beta)\der{g_j}{x}{}, \nonumber \\
    \sigma_i (x,\beta) &= \norm{g_{+}(x),h(x)}^{1-\beta}_\beta \left | h_i(x)\right |^{\beta -1} \text{sgn}(h_i(x)), \nonumber \\ 
    \eta_j(x,\beta) &= \norm{g_{+}(x),h(x)}^{1-\beta}_\beta  \left (g(x)_{+}\right )^{\beta - 1}_{j}, \nonumber
\end{align}
for $i\in E$ and $j\in I$. 

We state the adaptive algorithm utilizing dual space updates and a mirror mapping to the primal space:

\begin{algorithm}[H]
\caption{(MD for exact penalty function)}
\label{alg:smdp}
\begin{algorithmic}[1]
\Require Initial dual point $Y_0 \in \R^n$, $\kappa > 1$, $\beta \in (1,+\infty)$, initial penalty $p_0 > 0$, and stepsizes $\{\gamma_k\}$.
\For {$k = 0, 1, 2, \dots$}
    \State $X_k = \Mirror_h(Y_k)$
    \State $g_f = \der{f}{X_k}{}$, $g_\beta = g_\beta(X_k)$
    \State $p = p_k$
    \If {$M_\beta (X_k) > 0$}
        \While {$DP^\beta_p(X_k; -g_f - p g_\beta) + M_\beta (X_k)/p > 0$}
            \State $p = \kappa p$
        \EndWhile
    \EndIf
    \State $p_{k+1} = p$
    \State $Y_{k+1} = Y_k - \gamma_k (g_f + p_{k+1} g_\beta)$
\EndFor
\end{algorithmic}
\end{algorithm}

The while-loop evaluates the continuous directional derivative to verify if the penalty parameter enforces sufficient descent. To prove that the sequence $\{p_k\}$ remains bounded, we define the following constraint qualification adapted for the variationally coherent framework.



\begin{definition}[Constraint Qualification -- ${\bf (CQ_{VC})}$]\label{def:cqcv}
There exist $v\in \R^n$, $K_1 > 0$, and $K_2 > 0$ such that for all $x \in \R^n$, the constraint directional derivatives along the direction $d = v - x$ satisfy:
\begin{align}
    & \sgn(h_i(x)) \mul{\der{h_i}{x}{}}{d} \leq -K_1 \frac{|h_i(x)|}{M_\infty(x)}, \quad \forall i\in E, \label{CQVC1} \\
    & \mul{\der{g_j}{x}{}}{d} \leq -K_2, \quad \forall j\in I_+(x). \label{CQVC2}
\end{align}
where $I_{+}(x) = \{j\in I \mid g_j(x) > 0\}$.
\end{definition}

If functions $g_j$, $j\in I$ are convex, functions $h_i$, $i\in E$ are linear, and for the problem $\C$ the Slater constraint qualification holds, then (\ref{CQVC1})--(\ref{CQVC2}) will hold. Suppose that there exists $\hat{x}\in \R^n$ such that $g_j(\hat{x}) < 0$, $j\in I$ and $h_i(\hat{x}) = 0$, $i\in E$. Then $0 > g_j(\hat{x}) \geq g_j(x) + \mul{\der{g_j}{x}{}}{\hat{x}-x}, \quad j\in I$, which implies that $\mul{\der{g_j}{x}{}}{\hat{x} - x} \leq g_j(\hat{x}) - g_j(x), \quad j\in I$, and all $x\in \X$. In order to satisfy (\ref{CQVC2}) it is sufficient to take $K_2 = \min_{j\in I} \left [-g_j(\hat{x})\right ]$.

As far as the equality constraints are concerned we have: if $M_{E}^i(x) = |h_i(x)|$ then $M_{E}^i(x) \geq M_{E}^i(\hat{x}) = 0$ for any $x\in \X$. Function $M_E^i$ is convex since $h_i$ are linear and $|\cdot|$ is convex. Thus, we have $M_E^i(\hat{x}) \geq M_E^i(x) + \mul{g}{\hat{x}-x}$ for any $g\in \partial M_E^i(x)$. In particular, we have:
\begin{equation}
    -\frac{|h_i(x)|}{M_\infty(x)} \geq -|h_i(x)| \geq \sgn(h_i(x)) \mul{\der{h_i}{x}{}}{\hat{x}-x}.
\end{equation}
It follows that in order to satisfy (\ref{CQVC1}) we can take $K_1 = 1$.

For clarity of presentation, we assume sequence $\{X_k\}$ remains bounded. The more general case, where the sequence remains in the compact set $\X$ is presented in the following subsection.

\begin{lemma}\label{boundedPenalty}
Suppose that $f$ is continuously differentiable, ${\bf (A3)}$ is satisfied and  ${\bf (CQ_{VC})}$ holds. Assume that the sequence $\{X_k\}$ is bounded, then the sequence of penalty parameters $\{p_k\}$ generated by Algorithm \ref{alg:smdp} is bounded.
\end{lemma}

\begin{proof} 
Consider $X_k$, at which penalty is increased. Let $G_k = \der{f}{X_k}{} + p_k g_\beta(X_k)$. The continuous directional derivative in the search direction $-G_k$ is $DP^\beta_p(X_k; -G_k) = -\norm{G_k}^2$. If penalty is increased, it is necessary that 
\begin{equation}
    \norm{G_k}^2 < \frac{M_\beta(X_k)}{p_k}. \label{eq:while_bound}
\end{equation}
Let $d_k = v - X_k$, where $v$ is the fixed global reference point from ${\bf (CQ_{VC})}$.

By substituting the bounds from ${\bf (CQ_{VC})}$ into the penalty subgradient expansion, we have:
\begin{align*}
    \mul{g_\beta(X_k)}{d_k} &= \sum_{i \in E} |\sigma_i(X_k, \beta)| \, \xi_i(X_k, d_k) + \sum_{j \in I_+(X_k)} \eta_j(X_k, \beta) \, \zeta_j(X_k, d_k) \\ 
    &\leq -K_1 \sum_{i \in E} |\sigma_i(X_k, \beta)| \frac{|h_i(X_k)|}{M_\infty(X_k)} - K_2 \sum_{j \in I_+(X_k)} \eta_j(X_k, \beta).
\end{align*}

Let $K_3 = \min(K_1, K_2)$. Because the algorithm only tests the penalty update when $X_k$ is strictly infeasible, i.e., $M_\beta(X_k) > 0$, finite-dimensional norm equivalence guarantees that the summation of coefficients is strictly bounded below by a constant $\nu > 0$. Thus, $\mul{g_\beta(X_k)}{d_k} \leq -K_3 \nu$.

By the assumption that $\{X_k\}$ is bounded, and the objective function $f$ is smooth (Assumption {\bf (A1)}), the objective gradient $\der{f}{X_k}{}$ is bounded. Therefore, there exists a finite constant $\delta > 0$ such that $\mul{\der{f}{X_k}{}}{d_k} \le \delta$. Combining these bounds yields:
\begin{equation}
    \mul{G_k}{d_k} \le \delta - p_k K_3 \nu.
\end{equation}

By Cauchy-Schwarz we get $\mul{G_k}{d_k} \ge -\norm{G_k} \norm{d_k}$, and since $\{X_k\}$ is bounded and $v$ is fixed, there exists $D > 0$ bounding the distance from above $\norm{d_k} = \norm{v - X_k} \leq D$. Combining this with the previous inequality yields:
\begin{equation}
    -\norm{G_k} D \le \delta - p_k K_3 \nu \implies \norm{G_k} \ge \frac{p_k K_3 \nu - \delta}{D}.
\end{equation}

Substituting this lower bound for $\norm{G_k}$ back into the while-loop condition (\ref{eq:while_bound}) yields:
\begin{equation}
    \left( \frac{p_k K_3 \nu - \delta}{D} \right)^2 \le \norm{G_k}^2 < \frac{M_\beta(X_k)}{p_k},
\end{equation}
and multiplying both sides by $p_k D^2$ gives:
\begin{equation}
    p_k (p_k K_3 \nu - \delta)^2 < D^2 M_\beta(X_k).
\end{equation}

Since $\{X_k\}$ is bounded and the constraint functions are continuous, the violation $M_\beta(X_k)$ is bounded above by a finite constant $M_{\max}$. Thus, we obtain:
\begin{equation}
    p_k (p_k K_3 \nu - \delta)^2 < D^2 M_{\max},
\end{equation}
showing that $p_k$ cannot diverge to infinity.
\end{proof}

\begin{lemma}[Feasibility of the Limit Point]\label{lem:feasibility}
Suppose the assumptions of Lemma \ref{boundedPenalty} hold. Let $\{X_k\}$ be the sequence generated by Algorithm \ref{alg:smdp}, and assume it converges to a stationary point $\xo$ of the penalty function. Then, $\xo$ must be feasible, i.e., $M_\beta (\xo) = 0$.
\end{lemma}
\begin{proof}
By {\it Lemma \ref{boundedPenalty}}, the sequence of penalty parameters $\{p_k\}$ is strictly bounded. Because $p_k$ is only updated via multiplication by $\kappa > 1$, it must stabilize after a finite number of iterations at a fixed value $p_{\max}$. Thus, for all sufficiently large $k$, $p_k = p_{\max}$.

Assume, on the contrary, that the sequence $\{X_k\}$ converges to an infeasible stationary point $\xo$, meaning $M_\beta (\xo) > 0$. Because the $\beta$-norm formulation ensures $P_{p_{\max}}^\beta(x)$ is continuously differentiable at strictly infeasible points, its gradient must vanish at this stationary point: $\der{P_{p_{\max}}^\beta}{\xo}{} = 0$. 

Evaluating the limit of the continuous test function as $X_k \to \xo$ yields:
\begin{align*}
    \lim_{k \to \infty} \left[ DP_{p_{\max}}^\beta \left( X_k; -\der{f}{X_k}{} - p_{\max} g_\beta(X_k) \right) + \frac{M_\beta(X_k)}{p_{\max}} \right] &= DP_{p_{\max}}^\beta(\xo; 0) + \frac{M_\beta (\xo)}{p_{\max}}
    = \frac{M_\beta (\xo)}{p_{\max}}.
\end{align*}

Because $M_\beta (\xo) > 0$, this limit evaluates to a strictly positive constant. Therefore, for sufficiently large $k$, the test condition in Algorithm \ref{alg:smdp} evaluates to strictly greater than zero, triggering the update rule $p \leftarrow \kappa p$. This forces the sequence $\{p_k\}$ to diverge to infinity, which strictly contradicts the boundedness established in {\it Lemma \ref{boundedPenalty}}. Thus, the assumption $M_\beta (\xo) > 0$ is false, and the algorithm must converge to a point within the feasible set $\C$.
\end{proof}

\begin{theorem}[Global Convergence]\label{SMDPenaltyUpdate}
Suppose $f$ is continuously differentiable and assumptions {\bf (A3)}--{\bf (A5)} hold. 
If the constraint qualification ${\bf (CQ_{VC})}$ holds, the sequence $\{X_k\}$ generated by Algorithm \ref{alg:smdp} is bounded, and the problem $\Up$ is variationally coherent, then $\{X_k\}$ converges to a global minimizer of $\Up$, and this minimizer satisfies the constraints of problem $\C$. 
\end{theorem}
\begin{proof}
By {\it Lemma \ref{boundedPenalty}}, the sequence $\{p_k\}$ stabilizes at a finite value $p_{\max}$. Consequently, Algorithm \ref{alg:smdp} ultimately executes standard Stochastic Mirror Descent on the fixed, unconstrained penalty function $P_{p_{\max}}^\beta(x)$. By Theorem \ref{globalConvergence}, under variational coherence and the specified step size criteria, the sequence $\{X_k\}$ converges to a global minimizer $\xo$ of $P_{p_{\max}}^\beta$.

By {\it Lemma \ref{lem:feasibility}}, this minimizer $\xo$ cannot be strictly infeasible, as converging to an infeasible point forces $p_k \to \infty$, breaking the established bound. Therefore, $M_\beta(\xo) = 0$. Since $\xo$ is a global minimizer of the exact penalty function and is strictly feasible, it coincides exactly with a global minimizer of the original constrained problem $\C$.
\end{proof}

\subsection{Finite penalty parameter over a compact set}

Now we present the exact penalty mirror descent algorithm ({\it Algorithm \ref{alg:alg2}}) with compact set constraints. The algorithm differs in two important things. Firstly, the mirror map uses projection. Secondly, the test function tests the primal difference in trajectory, i.e., reduced gradient, to establish convergence.

To accommodate the strict domain boundaries of $\X$, the mirror map must intrinsically handle these projections. Within the Fenchel coupling framework, the mirror map $\Mirror_h: \R^n \to \X$ generated by a strictly convex regularizer $h(x)$ is defined via the convex conjugate as $\Mirror_h(y) = \der{h^*}{y}{} = \arg\max_{z \in \X} \left\{ \mul{y}{z} - h(z) \right\}$. 

Notice that when employing the standard Euclidean regularizer $h(x) = \frac{1}{2}\norm{x}_2^2$, the mirror map evaluation equates to the Euclidean projection onto the compact set $\X$. For an arbitrary dual update step $Y_{k+1} = Y_k - \gamma_k G_k$, the primal mapping resolves to (see (\ref{SMD_step1})--(\ref{SMD_step2})):
\begin{align}
    X_{k+1} = \Mirror_h(Y_{k+1}) &= \arg\min_{z \in \X} \left\{ \frac{1}{2}\norm{z}_2^2 - \mul{Y_{k+1}}{z} \right\} \nonumber \\
    &= \arg\min_{z \in \X} \left\{ \frac{1}{2} \norm{Y_{k+1} - z}_2^2 \right\} = \mathcal{P}_\X [Y_{k+1}], \label{eq:mirror_projection}
\end{align}
where $\mathcal{P}_\X$ denotes the Euclidean projection operator. This demonstrates a critical feature of the proposed method: the explicit projection operator is encapsulated entirely within the mirror map evaluation. Consequently, the updated primal point $X_{k+1}$ naturally satisfies the standard first-order optimality condition over the closed convex set $\X$:
\begin{equation} \label{eq:prox_optimality}
    \mul{\der{h}{X_{k+1}}{} - Y_{k+1}}{z - X_{k+1}} \ge 0, \quad \forall z \in \X.
\end{equation}

To understand why the test must change, recall the penalty update logic of the unconstrained method ({\it Algorithm \ref{alg:smdp}}). The algorithm evaluated the continuous directional derivative, effectively measuring the squared norm of the penalized subgradient: $\norm{G_k}^2 = \norm{\der{f}{X_k}{} + p g_\beta(X_k)}^2$. Because the exact penalty function is differentiable anywhere strictly outside the feasible set, a weak penalty parameter $p$ would cause the algorithm to converge to a false, infeasible minimum where the gradient naturally vanishes ($\norm{G_k}^2 \to 0$). The algorithm detected this vanishing gradient while $M_\beta(X_k) > 0$ and correctly increased $p$.

However, when optimization is restricted to a compact set $\X$, this subgradient-based logic fails. Suppose the penalty parameter is too weak, and the trajectory stalls against the boundary of the compact set $\X$ rather than finding the true feasible region. At this boundary, the negative descent direction $-(g_f + p g_\beta)$ points outside the set and is entirely absorbed by the normal cone $\mathcal{N}_\X(X_k)$. 

This creates a problem: the sequence of iterates stalls at the boundary, but the subgradient $G_k$ does not approach zero.
Instead, it approaches some non-zero vector $\nu \in \mathcal{N}_\X(X_k)$. Because the subgradient norm $\norm{G_k}^2$ remains strictly positive, the old test condition fails to recognize the stall. The penalty parameter is never updated, and the algorithm remains permanently trapped at an infeasible boundary point.

To restore convergence, the penalty test must measure the actual, physical movement of the algorithm. This is achieved by testing the primal difference in the trajectory---the reduced gradient---which inherently accounts for the geometry of the projection.

\subsubsection{The Primal Reduced Gradient and Design Choices}

To address the boundary stalling problem, the penalty update condition is formulated using the primal displacement rather than the subgradient. During a penalty update check at iteration $k$, we evaluate a trial step from the current state. The trial point is defined explicitly as a function of the dual variable $Y$ and the penalty parameter $p$, keeping the current gradients and step size $\gamma_k$ fixed:
\begin{equation} \label{eq:trial_step}
    X^+_k(Y, p) = \Mirror_h \Big( Y - \gamma_k ( \der{f}{X_k}{} + p g_\beta(X_k) ) \Big).
\end{equation}

Instead of testing the norm of the subgradient, the algorithm evaluates the primal reduced gradient of this trial step, denoted as $R_k(Y, p)$:
\begin{equation}
    R_k(Y, p) = \frac{1}{\gamma_k} \left( X_k - X^+_k(Y, p) \right).
\end{equation}
The use of $\norm{R_k(Y, p)}^2$ introduces geometric and computational tradeoffs that dictate the structure of the proposed algorithm.

An alternative to $R_k(Y, p)$ is the dual reduced gradient, defined as $\tilde{R}_k = (\der{h}{X_k}{} - \der{h}{X^+_k(Y, p)}{}) / \gamma_k$. While the dual reduced gradient offers robust theoretical bounding properties tied directly to the optimality conditions (\ref{eq:prox_optimality}), the primal reduced gradient is chosen for computational efficiency. The proposed algorithm operates within a lazy mirror descent framework, tracking the dual variable $Y_k$. Evaluating the dual reduced gradient requires explicitly computing the forward gradient $\der{h}{X^+_k(Y, p)}{}$ at every trial step, adding unnecessary computational overhead. By contrast, the primal difference $X_k - X^+_k(Y, p)$ is a direct byproduct of the mirror map evaluation.

To guarantee that bounding the dual sequence naturally bounds the primal sequence $R_k(Y, p)$, the regularizer $h(x)$ must be $L$-smooth on the compact set $\X$. This requirement highlights a fundamental dichotomy in handling constraints within mirror descent. In the broader landscape of constrained optimization, constraints are typically enforced via one of two paradigms: reparameterization or explicit projection. The reparameterization approach utilizes distance-generating functions that act as barriers---such as the logarithmic barrier for positive orthants or the entropy regularizer over a simplex. These barrier functions grow to infinity as the trajectory approaches the boundary, effectively confining the sequence to the interior of the domain and eliminating the need for an explicit projection step. However, by their mathematical nature, such barrier functions are inherently not $L$-smooth. 

Conversely, the proposed algorithm operates on the explicit projection paradigm. As established in (\ref{eq:mirror_projection}), by projecting the dual update back onto the simple constraints $\X$ via the mirror map, the algorithm retains the critical ability to reach the boundary---and exactly satisfy the constraints---in a finite number of steps. Because the boundary is managed directly by the projection operator rather than the barrier geometry of the regularizer, the distance-generating function does not need to exhibit barrier-like growth. Thus, selecting an $L$-smooth regularizer naturally aligns with the projection-based framework, allowing strict metric equivalence between the primal and dual spaces without generating theoretical conflicts.

The primal reduced gradient introduces a distinct edge case when operating with large step sizes. Consider an iteration where the current point $X_k$ is strictly infeasible. If a large step size $\gamma_k$ is applied, the unprojected dual update may point far outside the set $\X$, causing the mirror map to project the trial point $X^+_k(Y, p)$ onto a distant boundary of $\X$. While the physical displacement $\norm{X_k - X^+_k(Y, p)}$ is inherently bounded by the finite diameter of the compact set $\X$, the reduced gradient $R_k(Y, p)$ is inversely scaled by $\gamma_k$. Consequently, an oversized step size can force $\norm{R_k(Y, p)}^2$ to evaluate to an arbitrarily small scalar. 

This creates a scenario where the test condition $\norm{R_k(Y, p)}^2 < M_\beta(X_k) / p$ is improperly triggered, initiating a penalty increase even if the current penalty parameter is theoretically sufficient. However, this geometric artifact is algorithmically safe and does not compromise convergence. If the penalty parameter $p$ is updated inside the loop, the trial point $X^+_k(Y, p)$ may remain pinned against the boundary of $\X$, which is bounded, meaning the reduced gradient norm $\norm{R_k(Y, p)}$ is also bounded. Crucially, as $p$ is multiplied by $\kappa$, the right-hand side of the test condition, $M_\beta(X_k) / p$, strictly monotonically decreases toward zero. Therefore, the inequality will inevitably fail after a finite number of iterations. 
The algorithm terminates the loop after a finite number of parameter updates and resumes standard optimization.

\subsubsection{Algorithm}

With the primal reduced gradient defined and the regularizer assumptions established, we present the unified Exact Penalty Mirror Descent algorithm for compact sets ({\it Algorithm \ref{alg:alg2}}). 

The algorithm is structured to default to computationally efficient lazy mirror descent. At each iteration, the default trial step $X^+_k(Y_k, p)$ is evaluated using the accumulated dual variable $Y_k$. The algorithm then checks if the current state is strictly infeasible ($M_\beta(X_k) > 0$) and if the resulting primal reduced gradient is insufficient to guarantee descent ($\norm{R_k(Y_k, p)}^2 < M_\beta(X_k) / p$). 

If both conditions are met, the algorithm intercepts the step. Because the failure might be caused by an excessively large accumulated dual variable rather than an insufficient penalty parameter, the algorithm performs a dual reset---overwriting the dual variable with the exact inverse mirror map $\der{h}{X_k}{}$. It then enters a validation loop. Inside this loop, the step is re-evaluated from the anchored dual state. If the condition is satisfied, the loop terminates immediately without altering the penalty. Only if the descent condition fails again does the algorithm scale the penalty parameter by $\kappa$ and repeat the evaluation. Once the descent condition is met, the active dual variable is finalized, and the accepted trial step is directly assigned to $X_{k+1}$.

\begin{algorithm}[h]
\caption{(Exact Penalty Mirror Descent over a Compact Set)}
\label{alg:alg2}
\begin{algorithmic}[1]
\Require Initial dual point $Y_0 \in \R^n$, $\kappa > 1$, $\beta \in (1,+\infty)$, initial penalty $p_0 > 0$, and step sizes $\{\gamma_k\}$.
\State $X_0 = \Mirror_h(Y_0)$
\For {$k = 0, 1, 2, \dots$}
    \State $g_f = \der{f}{X_k}{}$, $g_\beta = g_\beta(X_k)$
    \State $p = p_k$
    \State $Y = Y_k$
    
    \State $X^+ = X^+_k(Y, p)$
    
    \If {$\norm{X_k - X^+}^2 / \gamma_k^2 < M_\beta(X_k) / p$}
        \State $Y = \der{h}{X_k}{}$
        \Loop
            \State $X^+ = X^+_k(Y, p)$
            \If {$\norm{X_k - X^+}^2 / \gamma_k^2 \ge M_\beta(X_k) / p$}
                \State \textbf{break}
            \EndIf
            \State $p = \kappa p$
        \EndLoop
    \EndIf
    
    \State $p_{k+1} = p$
    \State $Y_{k+1} = Y - \gamma_k (g_f + p_{k+1} g_\beta)$
    \State $X_{k+1} = X^+$
\EndFor
\end{algorithmic}
\end{algorithm}

\subsubsection{Theoretical Guarantees}

In order to prove convergence, the Constraint Qualification is adapted to include the simple set constraints. 

\begin{definition}[Constraint Qualification --- ${\bf (CQ_{VC} X)}$]\label{def:cqcvcompact}
Let $\X \subset \R^n$ be a compact set. There exist $v \in \X$, $K_1 > 0$, and $K_2 > 0$ such that for all $x \in \X$, the constraint directional derivatives along the direction $d = v - x$ satisfy:
\begin{align}
    & \sgn(h_i(x)) \der{h_i}{x}{d} \leq -K_1 \frac{|h_i(x)|}{M_\infty(x)}, \quad \forall i\in E, \label{CQVC1p} \\
    & \der{g_j}{x}{d} \leq -K_2, \quad \forall j\in I_+(x). \label{CQVC2p}
\end{align}
where $I_{+}(x) = \{j\in I \mid g_j(x) > 0\}$.
\end{definition}

We first establish that under ${\bf (CQ_{VC} X)}$, the exact penalty subgradient is uniformly bounded away from the descent point $v$.

\begin{lemma}[Subgradient Descent Bound] \label{lem:subgrad_bound}
Let $\X \subset \R^n$ be a compact set and assume ${\bf (CQ_{VC} X)}$ holds. Let $f$, $g_j$, and $h_i$ be continuously differentiable. For any strictly infeasible point $x \in \X$ (where $M_\beta(x) > 0$) and penalty parameter $p > 0$, the penalized subgradient $G(x) = \der{f}{x}{} + p g_\beta(x)$ satisfies: 
\begin{equation}
    \mul{G(x)}{v - x} \le \delta - p K_3 \nu
\end{equation}
where $\delta, K_3, \nu > 0$ are finite constants independent of $p$.
\end{lemma}
\begin{proof}
By substituting the bounds from ${\bf (CQ_{VC} X)}$ into the exact penalty subgradient expansion, the constraint violation gradient $g_\beta(x)$ guarantees a strict descent bound $\mul{g_\beta(x)}{v - x} \le -K_3 \nu$, where $K_3 = \min(K_1, K_2)$ and $\nu > 0$ is a constant derived from finite-dimensional norm equivalence. Because $\X$ is compact and $f$ is continuously differentiable, $\der{f}{x}{v - x}$ is strictly bounded from above by a finite constant $\delta$. The result follows immediately from the linearity of the inner product.
\end{proof}

We now prove that the exact penalty algorithm cannot get trapped in an infinite parameter update loop.

\begin{lemma}[Finiteness of the Penalty Parameter]\label{lem:boundedPenalty_MD}
Suppose the assumptions of Lemma \ref{lem:subgrad_bound} hold. Assume the regularizer $h(x)$ is $L$-smooth on the compact set $\X$. Then the sequence of penalty parameters $\{p_k\}$ generated by Algorithm \ref{alg:alg2} is strictly bounded.
\end{lemma}
\begin{proof}
Consider an iteration $k$ where the penalty parameter $p$ is increased inside the evaluation loop. The trial step $X^+ = X^+_k(\der{h}{X_k}{}, p)$ is evaluated from the reset dual state. By the definition of the mirror map, $X^+$ is the exact minimizer of the proximal projection, satisfying the variational inequality for all $z \in \X$ (Appendix \ref{app:variational_inequality}):
\begin{equation}
    \mul{\gamma_k G_k + \der{h}{X^+}{} - \der{h}{X_k}{}}{z - X^+} \ge 0.
\end{equation}
Substitute the global reference point $z = v \in \X$. By algebraically splitting the vector $v - X^+ = (v - X_k) + (X_k - X^+)$ and noting that $X_k - X^+ = \gamma_k R_k$, we obtain:
\begin{equation}
    \gamma_k \mul{G_k}{v - X_k} + \gamma_k^2 \mul{G_k}{R_k} + \mul{\der{h}{X^+}{} - \der{h}{X_k}{}}{v - X^+} \ge 0.
\end{equation}
Dividing by $\gamma_k$ and isolating the ${\bf (CQ_{VC} X)}$ term yields:
\begin{equation} \label{eq:vi_bound}
    -\mul{G_k}{v - X_k} \le \gamma_k \mul{G_k}{R_k} + \frac{1}{\gamma_k} \mul{\der{h}{X^+}{} - \der{h}{X_k}{}}{v - X^+}.
\end{equation}

We bound the terms on the right. By Cauchy-Schwarz, $\gamma_k \mul{G_k}{R_k} \le \gamma_k \norm{G_k} \norm{R_k}$. Because $\X$ is bounded and the problem functions are continuously differentiable, there exist constants $C_f, C_g > 0$ such that $\norm{G_k} \le C_f + p C_g$. Because $h(x)$ is $L$-smooth on $\X$, the dual norm is bounded: $\norm{\der{h}{X_k}{} - \der{h}{X^+}{}}_* \le L \norm{X_k - X^+} = L \gamma_k \norm{R_k}$. With the diameter of $\X$ bounded by $D$, the second term is strictly bounded by $L D \norm{R_k}$. 

Substituting these upper bounds and the lower bound from {\it Lemma \ref{lem:subgrad_bound}} into (\ref{eq:vi_bound}) yields:
\begin{equation}
    p K_3 \nu - \delta \le \Big( \gamma_k (C_f + p C_g) + L D \Big) \norm{R_k}.
\end{equation}
Because the penalty was increased, the test condition dictates $\norm{R_k} < \sqrt{M_\beta(X_k) / p}$. The constraint violation is bounded above by $M_{\max}$ over the compact set. Substituting $\norm{R_k} < \sqrt{M_{\max}} / \sqrt{p}$ yields:
\begin{equation}
    p K_3 \nu - \delta < \Big( \gamma_k C_f + \gamma_k p C_g + L D \Big) \frac{\sqrt{M_{\max}}}{\sqrt{p}}.
\end{equation}
Multiplying both sides by $\sqrt{p}$ isolates the dependency on the penalty parameter:
\begin{equation}
    (K_3 \nu) p^{3/2} - (\gamma_k C_g \sqrt{M_{\max}}) p - \delta p^{1/2} < \sqrt{M_{\max}} (\gamma_k C_f + L D).
\end{equation}
The right side of this inequality is a finite constant. However, as $p$ grows, the strictly positive $p^{3/2}$ term on the left side fundamentally dominates the lower-order terms. Therefore, for sufficiently large $p$, this inequality is mathematically impossible. The algorithm cannot satisfy the test condition infinitely, and the sequence $\{p_k\}$ is strictly finite.
\end{proof}

\begin{lemma}[Vanishing Reduced Gradient]\label{lem:vanishing_gradient}
Let $f$ be continuously differentiable over a compact set $\X$. Let $X_k = \Mirror_h(Y_k)$ be the sequence generated by lazy mirror descent:
\begin{equation}
    Y_{k+1} = Y_k - \gamma_k \der{f}{X_k}{},
\end{equation}
If $X_k \to \xo$, the primal reduced gradient $R_k = \frac{1}{\gamma_k}(X_k - X_{k+1})$ converges to zero.
\end{lemma}
\begin{proof}
Expanding the dual variable from $k=0$:
\begin{equation}
    Y_k = Y_0 - A_k,
\end{equation}
where $A_k = \sum_{i=0}^{k-1} \gamma_i \der{f}{X_i}{}$ is the accumulated gradient drift and $\tau_k = \sum_{i=0}^{k-1} \gamma_i$ is the total step size. 

Because $f$ is continuously differentiable and $X_k \to \xo$, the gradient sequence converges: $\lim_{k \to \infty} \der{f}{X_k}{} = \der{f}{\xo}{}$. Since $\tau_k \to \infty$, the Toeplitz Lemma (generalized Ces\`aro mean) guarantees that the weighted average of the gradients converges exactly to the same limit:
\begin{equation}
    \lim_{k \to \infty} \frac{A_k}{\tau_k} = \der{f}{\xo}{}.
\end{equation}

By the first-order optimality conditions of the mirror map for $X_k$ and $X_{k+1}$, and the $K$-strong convexity of $h$, we obtain:
\begin{equation}
    \mul{\der{h}{X_k}{} - \der{h}{X_{k+1}}{}}{X_k - X_{k+1}} \le \gamma_k \mul{\der{f}{X_k}{}}{X_k - X_{k+1}}.
\end{equation}
Applying strong convexity to the left side yields $K \norm{X_k - X_{k+1}}^2 \le \gamma_k \mul{\der{f}{X_k}{}}{X_k - X_{k+1}}$. Dividing by $\gamma_k^2$ isolates the reduced gradient $R_k$:
\begin{equation} \label{eq:gn_bound_diff}
    K \norm{R_k}^2 \le \mul{\der{f}{X_k}{}}{R_k}.
\end{equation}

By the Cauchy-Schwarz inequality, the bound implies $K \norm{R_k}^2 \le \norm{\der{f}{X_k}{}} \norm{R_k}$, which reduces to $\norm{R_k} \le \frac{1}{K} \norm{\der{f}{X_k}{}}$. Because $f$ is continuously differentiable over the compact set $\X$, its gradient is globally bounded by some constant $L > 0$ for all $x \in \X$. Therefore, the sequence of reduced gradients is strictly bounded:
\begin{equation} \label{eq:gn_sup_bound}
    \norm{R_k} \le \frac{L}{K} \quad \forall k.
\end{equation}
Note that this bound is independent of $k$ and step size.
The optimality condition for $X_k = \Mirror_h(Y_k)$ evaluated at $Z = X_{k+1} \in \X$ is:
\begin{equation}
    \mul{Y_k - \der{h}{X_k}{}}{X_{k+1} - X_k} \le 0 \implies 0 \le \mul{Y_k - \der{h}{X_k}{}}{R_k}.
\end{equation}
Substituting the dual expansion $Y_k = Y_0 - A_k$:
\begin{equation}
    0 \le \mul{Y_0 - A_k - \der{h}{X_k}{}}{R_k}.
\end{equation}
Dividing the entire inequality by $\tau_k$:
\begin{equation}
    0 \le \mul{\frac{Y_0 - A_k - \der{h}{X_k}{}}{\tau_k}}{R_k}.
\end{equation}
Adding $\mul{\der{f}{X_k}{}}{R_k}$ to both sides and grouping the right-hand terms yields:
\begin{equation}
    \mul{\der{f}{X_k}{}}{R_k} \le \mul{\der{f}{X_k}{} - \frac{A_k}{\tau_k} + \frac{Y_0 - \der{h}{X_k}{}}{\tau_k}}{R_k}.
\end{equation}
Chaining this with the bound from (\ref{eq:gn_bound_diff}) forms the final inequality:
\begin{equation}
    0 \le K \norm{R_k}^2 \le \mul{\der{f}{X_k}{}}{R_k} \le \mul{\der{f}{X_k}{} - \frac{A_k}{\tau_k} + \frac{Y_0 - \der{h}{X_k}{}}{\tau_k}}{R_k}.
\end{equation}
Taking the limit as $k \to \infty$: $\tau_k \to \infty$ and $\der{h}{X_k}{}$ is bounded over the compact set $\X$, so $\der{h}{X_k}{} / \tau_k \to 0$. We know $A_k / \tau_k \to \der{f}{\xo}{}$, and from spatial convergence, $\der{f}{X_k}{} \to \der{f}{\xo}{}$. Therefore, the grouped difference $\der{f}{X_k}{} - \frac{A_k}{\tau_k} \to \der{f}{\xo}{} - \der{f}{\xo}{} = 0$. 

Consequently, the entire vector on the right side of the inner product converges to zero. Because $\norm{R_k}$ is strictly bounded, the inner product must vanish:
\begin{equation}
    \lim_{k \to \infty} K \norm{R_k}^2 = 0 \implies \lim_{k \to \infty} \norm{R_k} = 0,
\end{equation}
which completes the proof. \end{proof}

With this behavior established, we can now prove the final global convergence theorem for the proposed algorithm over a compact set.

\begin{theorem}[Global Convergence]\label{thm:ProjectionPenaltyUpdate}
Suppose $f$ is continuously differentiable, assumption ${\bf (A3)}$ holds, the set $\X$ is compact, and ${\bf (CQ_{VC} X)}$ is satisfied. Assume the step size sequence $\{\gamma_k\}$ satisfies the Robbins-Monro conditions: $\sum_{k=0}^\infty \gamma_k = \infty$ and $\sum_{k=0}^\infty \gamma_k^2 < \infty$. If the exact penalty function $P^{\beta}_p$ satisfies variational coherence with respect to its optimal set, then there exists a subsequence of the sequence $\{X_k\}$ generated by Algorithm \ref{alg:alg2} which converges to the global solution of the constrained problem $\C$.
\end{theorem}
\begin{proof}
By {\it Lemma \ref{lem:boundedPenalty_MD}}, the sequence of penalty parameters $\{p_k\}$ is strictly bounded from above. Because the algorithm only updates the penalty via multiplication by a constant $\kappa > 1$, the parameter must stabilize after a finite number of iterations at a fixed value $p_{\max}$. 

For all subsequent iterations, {\it Algorithm \ref{alg:alg2}} executes standard mirror descent on the fixed, Lipschitz-continuous objective function $\Phi(x) = f(x) + p_{\max} M_\beta(x)$. Under the condition that $\Phi(x)$ is variationally coherent, the framework established for non-convex mirror descent guarantees convergence to a stationary point $\xo$ that minimizes $\Phi(x)$ globally.

We must verify that this limit point $\xo$ is feasible. Assume for the sake of contradiction that $p_{\max}$ is too low, causing the sequence to converge to an infeasible point ($M_\beta(\xo) > 0$). By {\it Lemma \ref{lem:vanishing_gradient}}, because the sequence converges to an infeasible stationary point under a fixed penalty, the primal reduced gradient vanishes ($\norm{R_k} \to 0$).

We now evaluate the algorithm's penalty update condition as $k \to \infty$:
\begin{equation}
    \norm{R_k}^2 < \frac{M_\beta(X_k)}{p_{\max}}.
\end{equation}
As $X_k \to \xo$, the left side evaluates to $0$. Because we assumed $\xo$ is strictly infeasible, the constraint violation on the right side converges to a strictly positive constant: $M_\beta(\xo) / p_{\max} > 0$. 

Therefore, in a finite number of steps, the condition $0 < M_\beta(\xo) / p_{\max}$ will be strictly satisfied, triggering the update $p \to \kappa p$. This directly contradicts the established fact that $p_k$ stabilizes at $p_{\max}$. 

Thus, our assumption must be false: $p_{\max}$ cannot be too low, and the sequence cannot converge to an infeasible point. The limit point $\xo$ must be strictly feasible ($M_\beta(\xo) = 0$). Because $\xo$ is the global minimizer of the exact penalty function and lies entirely within the feasible region, the exactness properties of the penalty formulation guarantee that $\xo$ is the global solution of the original constrained problem $\C$, completing the proof.
\end{proof}

\subsection{On Possible Extension to Stochastic Settings}
While this paper establishes the adaptive penalty update scheme for exact penalty mirror descent in a deterministic setting, a natural extension is its application to stochastic optimization, where only noisy gradient estimates $G_k = \der{f}{X_k}{} + U_k$ are available.

The introduction of stochasticity adds significant complexity to this penalty update mechanism. In the presence of persistent variance, the true reduced gradient $\norm{R_k}$ will not converge to zero when the algorithm stalls at an infeasible boundary. Instead, it will perpetually oscillate, bounded by the noise level $\norm{U_k}$. Because the penalty update mechanism strictly evaluates if $\norm{R_k}^2 < M_\beta(X_k)/p$, raw stochastic noise can trigger false penalty updates, causing $p_k$ to diverge and destroying the convergence guarantees of the algorithm.

To resolve this without altering the fundamental penalty update logic, the algorithm must dynamically control the variance of the gradient estimator by adjusting the minibatch size. Specifically, one can construct an unbiased estimator for the squared norm of the true gradient by evaluating the inner product of independent stochastic samples $G_k^{(i)}$ and $G_k^{(j)}$ drawn at the same point. Because the noise is zero-mean and independent, their expected inner product perfectly isolates the true squared norm: $\mathbb{E}[\langle G_k^{(i)}, G_k^{(j)} \rangle] = \norm{\der{f}{X_k}{}}^2$. By computing these inner products over an increasing minibatch of independent observations, one can construct high-probability confidence intervals $[L_k, U_k]$ around the true deterministic gradient norm, under basic assumptions on the noise distribution.

The penalty test mechanism can then be adapted to act strictly on statistical certainty:
\begin{enumerate}
    \item If the upper bound falls below the threshold, $U_k < \frac{M_\beta(X_k)}{p}$, the algorithm is statistically certain that the true reduced gradient has vanished. The penalty parameter $p$ is increased.
    \item If the lower bound exceeds the threshold, $L_k \ge \frac{M_\beta(X_k)}{p}$, the algorithm is statistically certain the penalty is sufficient, and the standard mirror descent step is accepted.
    \item If the threshold lies strictly inside the confidence interval ($L_k < \frac{M_\beta(X_k)}{p} < U_k$), the test is inconclusive. The algorithm must draw more observations to shrink the interval $[L_k, U_k]$ until one of the decisive conditions is met.
\end{enumerate}

Importantly, dynamically controlling the gradient variance in this manner is not expected to significantly worsen the algorithm's overall performance. Near convergence, when the true gradient naturally vanishes, a standard stochastic estimator often becomes completely indistinguishable from zero due to the noise. In these regions, increasing the minibatch size actively maintains a reasonable signal-to-noise ratio, successfully revealing the true descent direction that stochasticity would otherwise hide.

By wrapping the penalty update in this dynamic sampling regime, the algorithm can effectively emulate the deterministic boundary detection while safely accommodating stochastic oracles. While implementing this framework requires simple yet important extensions to the standard stochastic methodology presented earlier, the formal convergence analysis of this dynamic-batching update scheme is deferred to future work.

\section{Numerical Results}

In order to validate the method and, most importantly, to highlight some practical details, we evaluate the method numerically on selected examples. Firstly, we show example convergence trajectories for four different 2D non-convex benchmark objectives using a constant penalty parameter. Next, we evaluate the method on the Rosenbrock's function. Finally, we analyze the adaptive penalty update strategy and apply the algorithm to a binary regression task.

\subsection{Convergence for Constant Penalty on Elementary Examples}

To highlight the trajectory behavior of the stochastic mirror descent algorithm, we first evaluate the method using a constant, preselected penalty parameter. Consider the following elementary objective functions:
\begin{align*}
    f_{\rm{a}}(x_1, x_2) = &\ x_1^2 x_2^2 \\
    f_{\rm{b}}(x_1, x_2)= &
    \left[ 1 + (x_1+x_2+1)^2 (19 - 14x_1 + 3x_1^2 - 14x_2 + 6x_1x_2 + 3x_2^2) \right]
    \\ \cdot &
    \left[ 30 + (2x_1-3x_2)^2 (18 - 32x_1 + 12x_1^2 + 48x_2 - 36x_1x_2 + 27x_2^2) \right] \\
    f_{\rm{c}}(x_1, x_2) = &\ 100 \sqrt{\abs{x_2 - 0.01x_1^2}} + 0.01\abs{x_1 + 10} \\
    f_{\rm{d}}(x_1, x_2) = & (1.5 - x_1 + x_1 x_2)^2 + (2.25 - x_1 + x_1x_2^2)^2 + (2.625 - x_1 + x_1 x_2^3)^2
\end{align*}
which we call the bivariate quadratic product (a), Goldstein-Price (b), Bukin No. 6 (c), and Beale's (d) functions, respectively. 

Each optimization problem is constrained by a convex bounding box $\X$ and an exact penalty term $p M(x)$. The specific constraint violation functions $M(x)$ applied to each benchmark are defined as follows:
\begin{eqnarray}
&{\displaystyle 
    M_{\rm{a}}(x) = |x_1 - x_2|, \  M_{\rm{b}}(x) = |x_2 + 0.5|,} \nonumber \\ 
&{\displaystyle
    M_{\rm{c}}(x) = \max(0, 0.3 - x_2) + \max(0, -x_1 - x_2 - 1) + }\nonumber \\
&{\displaystyle 
    \max(0, x_1 - x_2 - 1), \  
    M_{\rm{d}}(x) = |x_1^2 + x_2^2 - 4|.}\nonumber 
\end{eqnarray}

Figure \ref{fig:trajectories} illustrates the resulting optimization trajectories (orange lines). For each test, the algorithm initializes at a starting point (white $+$), projects onto the boundary of the set $\X$ (red rectangle) when necessary, and successfully navigates toward the feasible region (green areas/lines indicating $M(x) = 0$). Despite the significant variance introduced by the stochastic gradient oracle, the last iterate (white $\times$) reliably converges to a local minimum within the feasible set. These baseline tests confirm that as long as the constant penalty parameter is chosen to be sufficiently large relative to the objective, the algorithm effectively isolates the true constrained solution.

\begin{figure}[ht]
    \centering
    \subfloat[Bivariate Quadratic]{\includegraphics[width=0.48\linewidth]{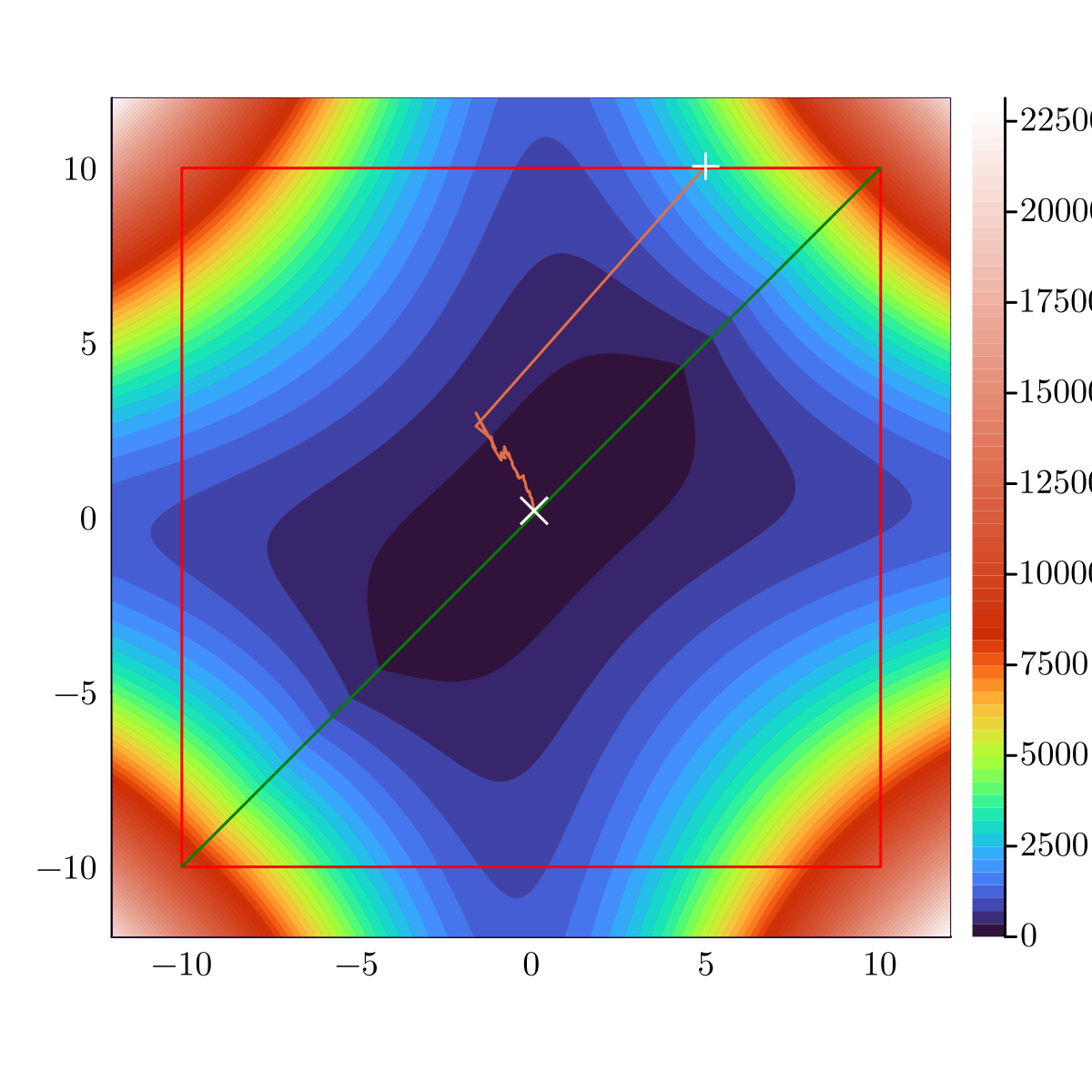}\label{fig7:a}} \hfill
    \subfloat[Goldstein-Price]{\includegraphics[width=0.48\linewidth]{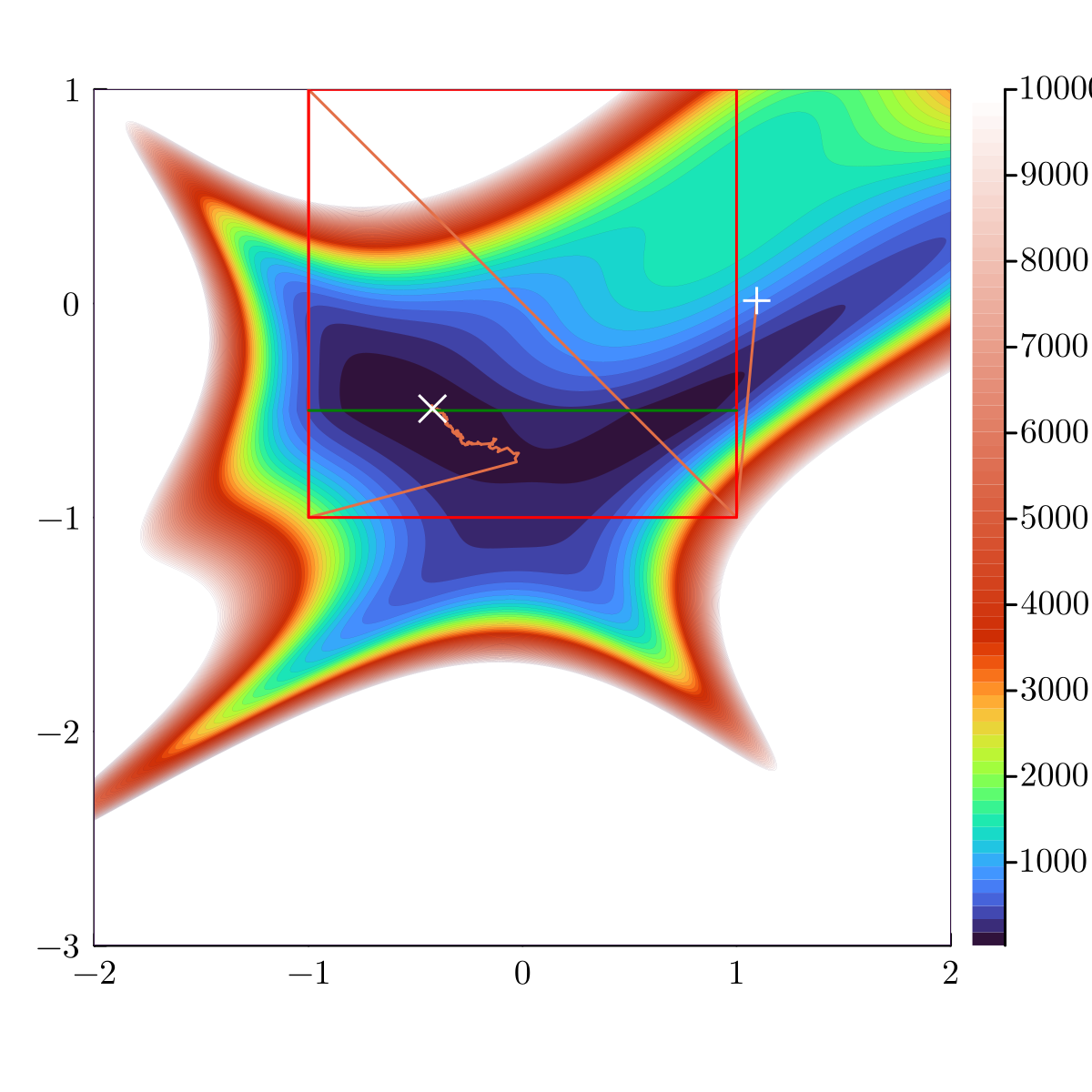}\label{fig7:b}} \\
    \subfloat[Bukin No. 6]{\includegraphics[width=0.48\linewidth]{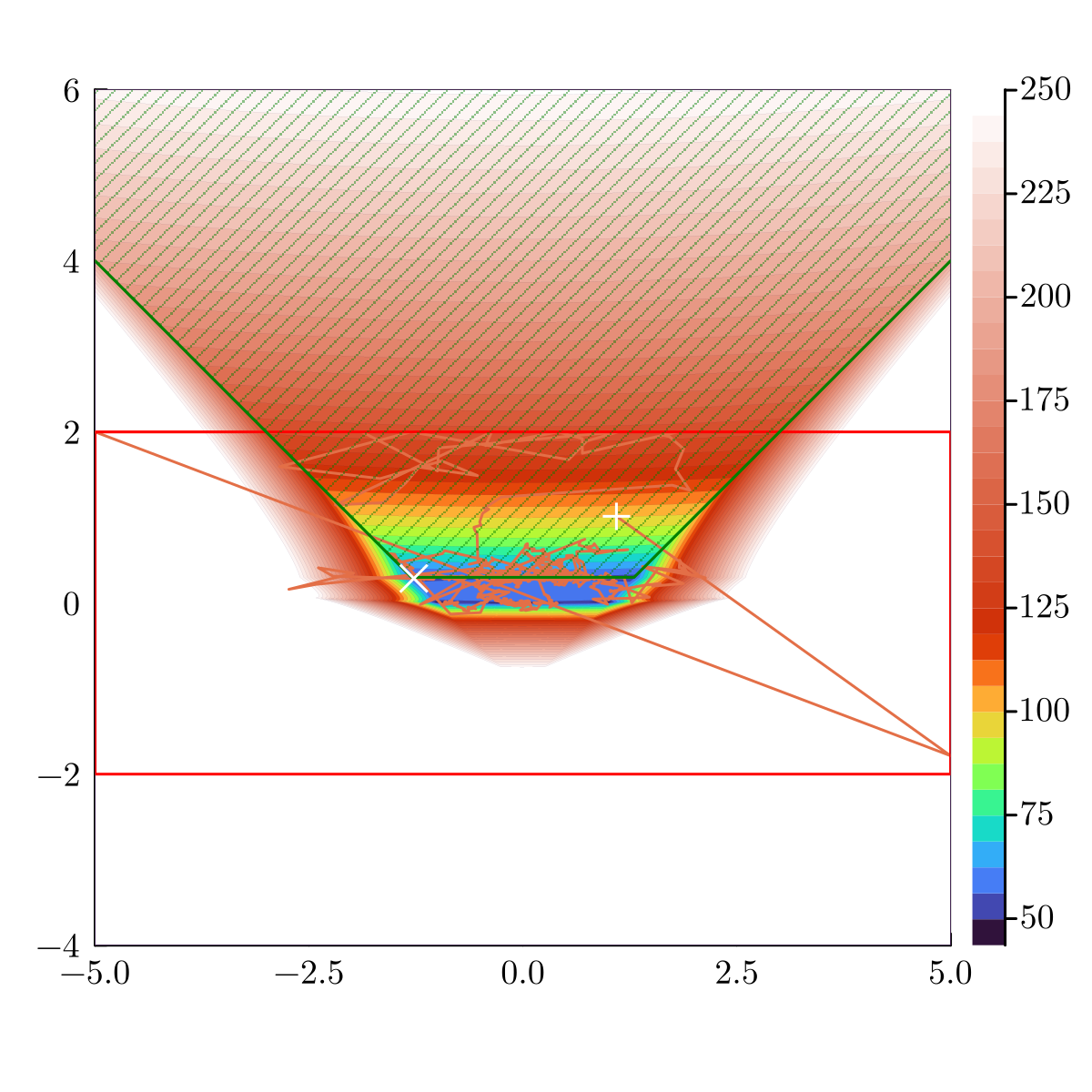}\label{fig7:c}} \hfill
    \subfloat[Beale]{\includegraphics[width=0.48\linewidth]{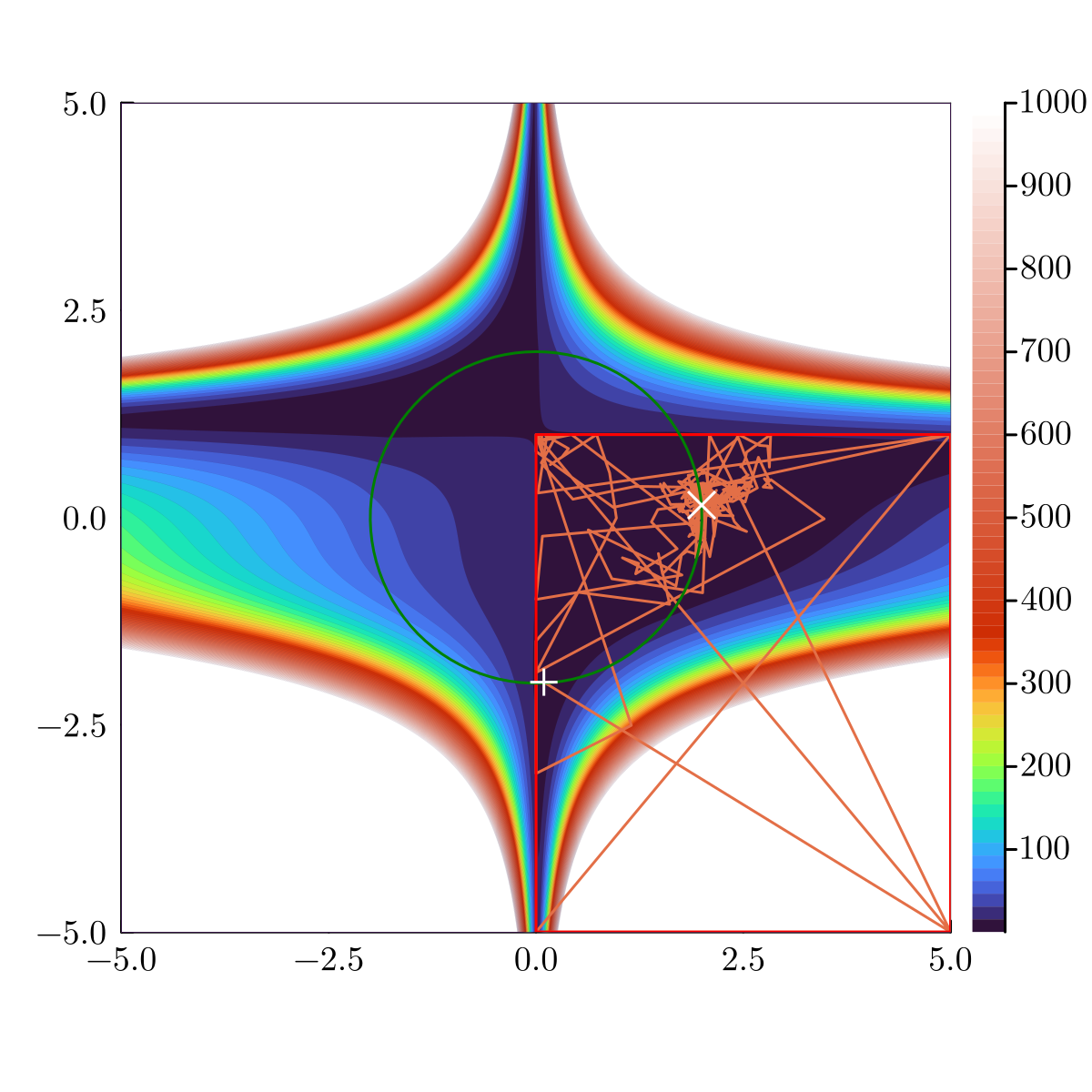}\label{fig7:d}}
    \caption{Optimization trajectories for 2D benchmarks using a constant exact penalty parameter. The algorithm effectively converges to the feasible sets (green) while remaining bounded by $\X$ (red).}
    \label{fig:trajectories}
\end{figure}

\subsection{Rosenbrock's function}

The Rosenbrock's function defined as:
\begin{equation}
    f_{\rm{Rosenbrock}}(x) =
    \sum_{i=1}^{n-1} \left[
        100(x_{i+1} - x_i^2)^2 + (1 - x_i)^2
    \right]
\end{equation}
where $x \in \R^n$, $n \in N$, is a common benchmark function for various optimization problems. The function can be optimized stochastically, by randomly selecting a subset of terms in the sum. In our case, we select only one sample per call to the oracle. We introduce the additional penalty function:
\begin{equation}
    p_{\rm{Rosenbrock}}(x) = \abs{x^T x - n}        
\end{equation}
and project each iterate onto the set $\X_{\rm{Rosenbrock}} = \overline{\B(0, 2\sqrt{n})}$. In this example we also use dual-averaging to speed up convergence.

Fig.~\ref{fig:rosenbrock} shows convergence of the algorithm for the Rosenbrock's function. The convergence is slow, but the algorithm converges after many iterations. Importantly, each call to the stochastic oracle gives only new information for two coordinates. Hence for large $n$, the convergence is very slow due to the nature of the oracle.

\begin{figure}[ht]
    \centering
    \subfloat[$n=4$]{
        \includegraphics[width=0.45\linewidth]{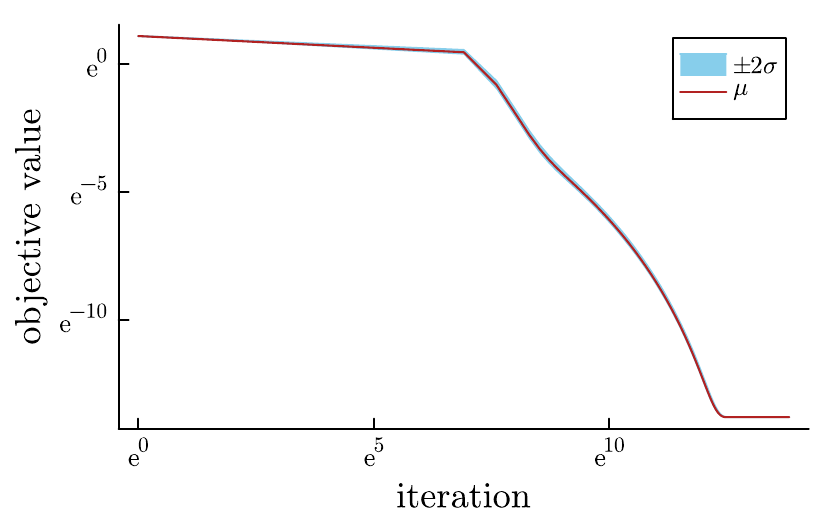}
        \label{fig8:a} 
    } 
    \subfloat[$n=8$]{
        \includegraphics[width=0.45\linewidth]{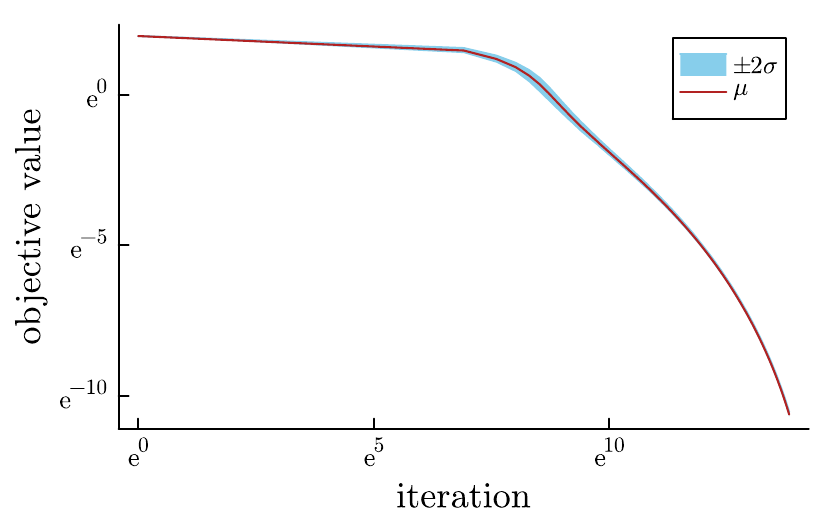}
        \label{fig8:b} 
    }
    \vfill
    \subfloat[$n=16$]{
        \includegraphics[width=0.45\linewidth]{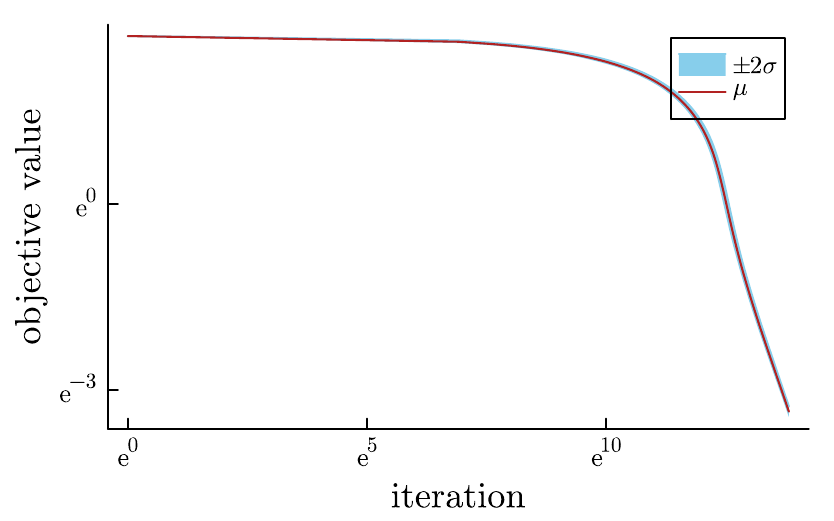}
        \label{fig8:c} 
    }
    \subfloat[$n=32$]{
        \includegraphics[width=0.45\linewidth]{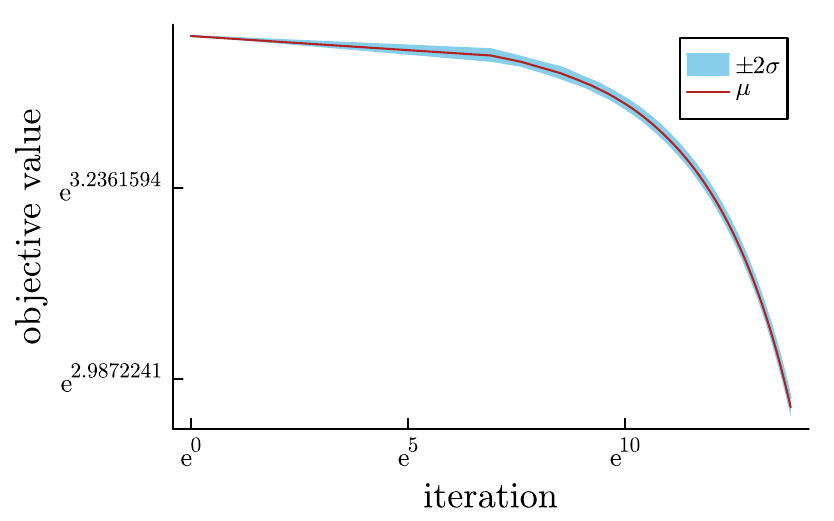}
        \label{fig8:d} 
    }
    \caption{
        Convergence example for Rosenbrock's function
        $f(x) = \sum_{i=1}^{n-1} [100(x_{i+1} - x_i^2)^2 + (1 - x_i)^2]$,
        for $x \in \R^n$ and $n \in \{4, 8, 16, 32\}$. Both axes are scaled logarithmically. The mirror function is an identity function. The stochastic gradient oracle takes the gradient of one random index at each step.}
    \label{fig:rosenbrock}
\end{figure}

\subsection{Adaptive Penalty Update Strategy (1D)}

To illustrate the mechanics of the adaptive update rule ({\it Algorithm 1}), Figure \ref{fig:penalty_update} visualizes the optimization of a 1D quadratic objective subject to an active inequality constraint, where the unconstrained minimum lies strictly outside the feasible set. The optimization process is characterized by two distinct phases. 

Initially, the penalty parameter $p_k$ is too small, and the iterates are drawn toward the infeasible unconstrained minimum. By evaluating the test function $t_{p_k}(X_k)$ outside the feasible set, the algorithm detects that the descent is insufficient and systematically increases $p_k$. The top plot of Figure \ref{fig:penalty_update} demonstrates how this adaptivity dynamically reshapes the exact penalty function $P_{p_k}(x) = f(x) + p_k M(x)$ (color-coded transitioning from purple to yellow) until the minimum is forced toward the feasible boundary.

Once $p_k$ crosses the critical threshold $\bar{p}$, the second phase begins. The algorithm converges to the sharp exact minimum at the true constrained solution. This sharp non-smoothness causes the subgradient norm to oscillate persistently (middle plot). This oscillation is a critical feature, as it safely prevents the test function from triggering any further, unnecessary increases to the penalty parameter, confirming the stabilization guaranteed by {\it Lemma \ref{boundedPenalty}}.

\begin{figure}[ht]
    \centering
    \includegraphics[width=0.7\linewidth]{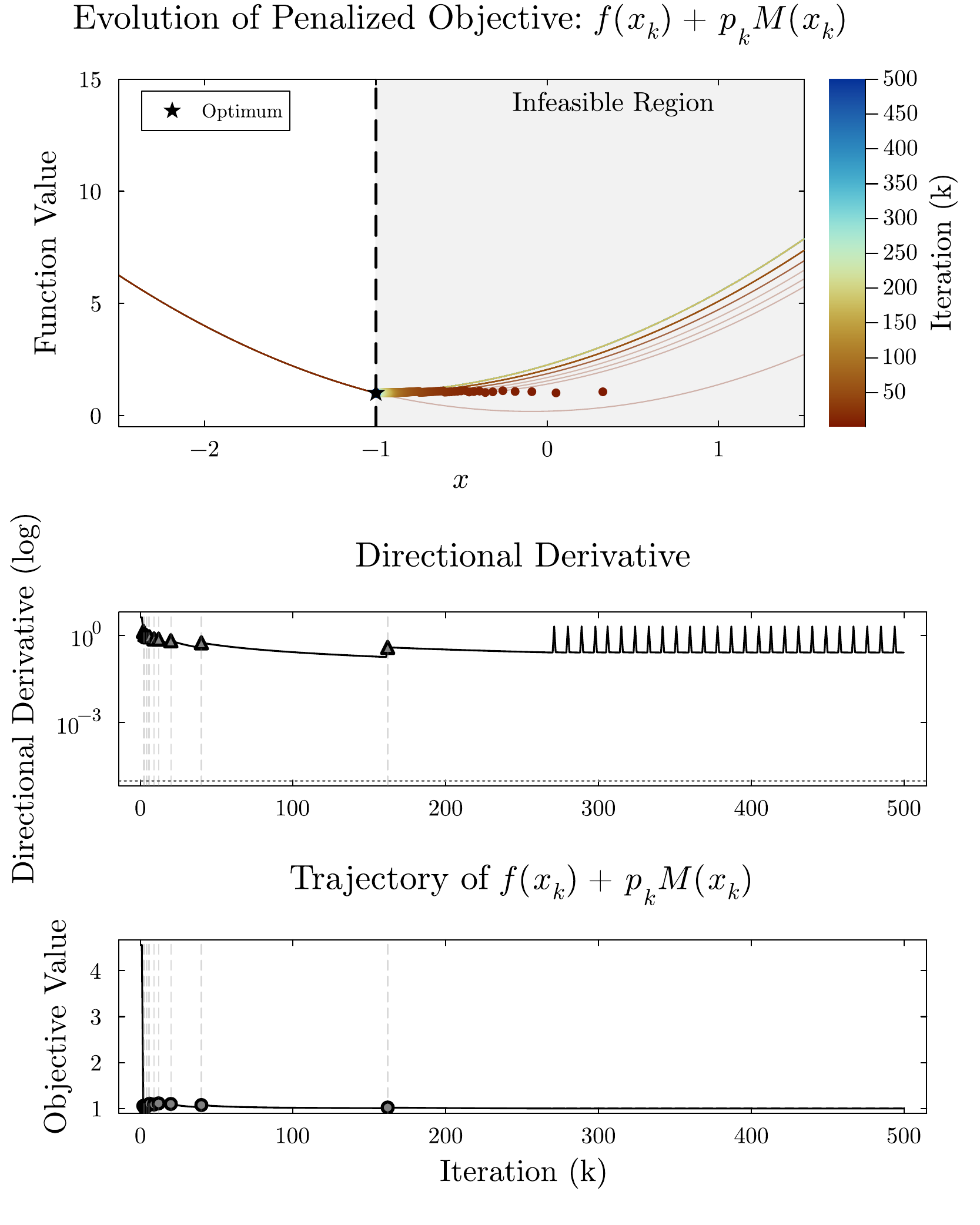}
    \caption{Evolution of the 1D adaptive penalty method for $f(x) \coloneq x^2$ and $M(x) \coloneq \max(0, x + 1)$. \textbf{Top:} The penalty function $P_{p_k}(x)$ reshaping as $p_k$ increases (purple to yellow). \textbf{Middle:} (Negative) directional derivative against iterations, oscillating at the non-smooth minimum. \textbf{Bottom:} Convergence of the total penalty function value.}
    \label{fig:penalty_update}
\end{figure}

\subsection{Adaptive Penalty Update Over a Compact Set (1D)}

To empirically validate the necessity of the primal reduced gradient (pseudo-gradient) introduced in Section 4.2, we evaluate a 1D optimization scenario constrained by both a general inequality constraint (e.g., $x \le -1$) and a simple compact set boundary requiring explicit projection (e.g., $x \ge -0.5$).

Recall the theoretical vulnerability of the standard subgradient test: if the trajectory stalls against the boundary of the simple set $\X$, the descent direction is absorbed by the normal cone. Figure \ref{fig:pseudo_gradient_1d} visualizes this exact pathology and demonstrates how the primal reduced gradient resolves it.

\begin{figure}[ht]
    \centering
    \includegraphics[width=0.61\linewidth]{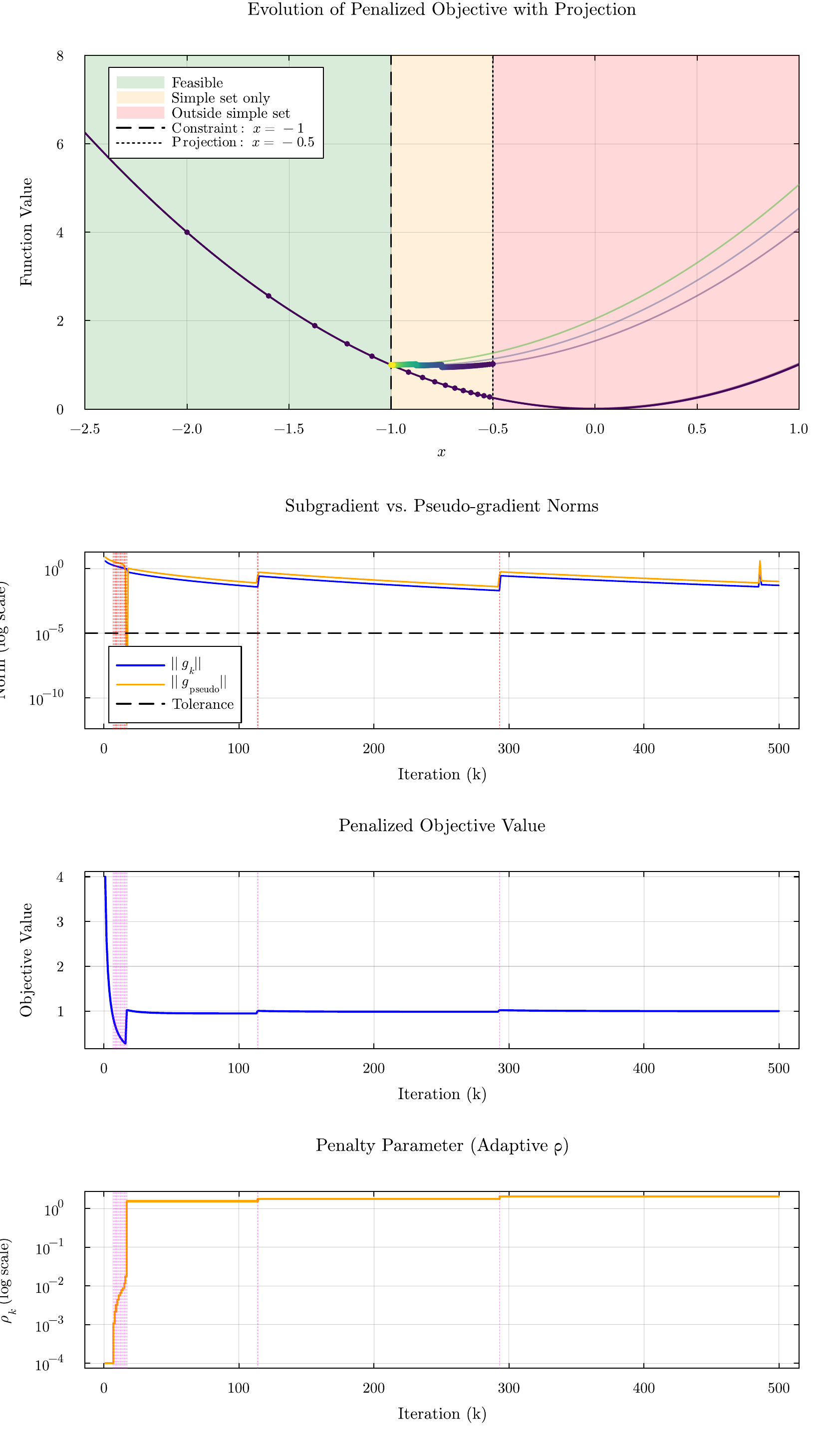}
    \caption{Evolution of the adaptive penalty method over a simple set boundary. \textbf{Top:} The penalized objective overlaid with the simple set boundaries and constraints. \textbf{Second:} Comparison of the analytical subgradient norm ($\|g_k\|$) versus the pseudo-gradient norm ($\|g_{pseudo}\|$). \textbf{Third:} Convergence of the penalized objective value. \textbf{Bottom:} The step-wise inflation of the adaptive penalty parameter $p_k$.}
    \label{fig:pseudo_gradient_1d}
\end{figure}

The optimization process illustrated in Figure \ref{fig:pseudo_gradient_1d} highlights the critical behavioral divergence between the analytical subgradient and the primal reduced gradient:

\begin{itemize}
    \item \textbf{The Boundary Stall:} Initially, an insufficient penalty parameter draws the iterates toward an infeasible unconstrained minimum. As the trajectory attempts to exit the simple set $\X$ to reach this minimum, the mirror map projects the iterate onto the boundary (e.g., at $x = -1.0$). 
    \item \textbf{Subgradient Failure vs. Reduced Gradient Detection:} As shown in the second panel, when the sequence stalls at the boundary, the norm of the analytical subgradient $\norm{G_k}$ remains strictly positive because the descent direction continues to point outward. If the algorithm relied on the standard subgradient test, it would fail to recognize the stall and subsequently fail to update $p_k$. In contrast, the primal reduced gradient $\norm{R_k}$ measures the effective displacement after projection. Because the projection restricts the iterate to the boundary, this displacement evaluates to zero, causing $\norm{R_k}$ to fall below the tolerance threshold.
    \item \textbf{Successful Adaptive Update:} Because the reduced gradient correctly identifies the stall ($\norm{R_k}^2 < M_\beta(X_k) / p$), the algorithm triggers the penalty update condition. As illustrated in the bottom panel, the penalty parameter $p_k$ is systematically increased. This dynamically reshapes the penalized objective landscape (top panel) until the required descent direction aligns with the feasible region, allowing the algorithm to resume convergence.
\end{itemize}

This experiment confirms that replacing the subgradient with the primal reduced gradient is not merely a theoretical convenience, but an algorithmic necessity to ensure convergence when exact penalty methods are subjected to compact set projections.

\subsection{The Necessity of Differentiability Outside the Feasible Set}

The adaptive penalty update mechanism relies on the gradient vanishing when the sequence converges to an infeasible unconstrained minimizer. Figure \ref{fig:significance_of_differentiability} illustrates why standard non-smooth penalties (e.g., $\ell_1$) fail to satisfy this condition, whereas the proposed $\beta$-norm formulation succeeds.

Consider the minimization of $f(x) = x_1 + x_2$ subject to $x_1 \leq 0$ and $x_2 - x_1 \leq 0$. A standard non-smooth penalty introduces points of non-differentiability strictly outside the feasible set (Figure \ref{fig:significance_of_differentiability}, right). At these points, the norm of the subgradient remains bounded away from zero (orange line). Consequently, the penalty update condition is never satisfied, $p_k$ remains constant, and the method fails to reach the feasible set. In contrast, the $\beta$-norm formulation guarantees continuous differentiability at all strictly infeasible points (Figure \ref{fig:significance_of_differentiability}, left). The gradient vanishes at the infeasible minimizer (blue line), which successfully triggers the update condition to increase $p_k$ and directs the sequence toward the optimal solution of the constrained problem.

\begin{figure}[ht]
    \centering
    \includegraphics[width=0.99\linewidth]{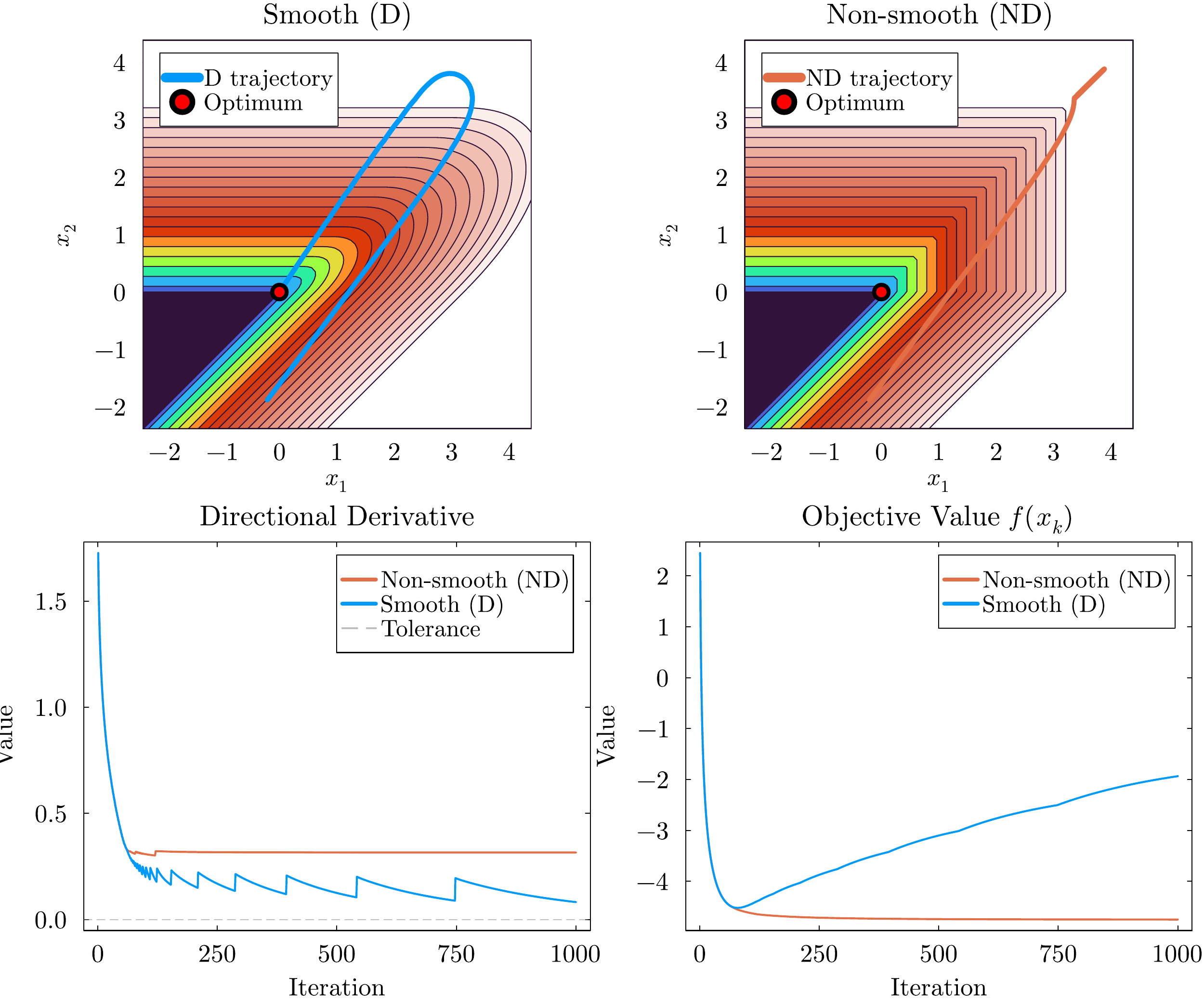}
    \caption{Comparison of the adaptive algorithm using a continuously differentiable penalty outside the feasible set ($\beta$-norm, blue lines) versus a standard non-smooth penalty ($\ell_1$, orange lines). The $\beta$-norm formulation ensures the gradient vanishes at infeasible stationary points, guaranteeing the execution of the required penalty updates.} 
    \label{fig:significance_of_differentiability}
\end{figure}

\subsection{Binary Regression}

To evaluate the algorithm on a problem requiring simultaneous continuous optimization and discrete feature selection, we apply it to a relaxed binary regression task. The goal is to recover a sparse, strictly discrete weight vector $w^* \in \mathbb{R}^p$, where the active features are drawn with a probability of 0.3, $w^*_i \sim \text{Bernoulli}(0.3)$. 

The core optimization problem is formulated as a least-squares regression:
\begin{equation}
    \min_{w \in \mathbb{R}^p} \quad f(w) \coloneq \|X_{\text{train}} w - y_{\text{train}}\|_2^2 
\end{equation}
where $X_{\text{train}}$ contains standard Gaussian features, and the continuous target variable is generated as $y = X w^* + \varepsilon$ with Gaussian noise $\varepsilon \sim \mathcal{N}(0, 0.01 I)$.

\subsubsection*{Exact Penalty Formulation and Optimization}

To force the continuous weights toward the desired discrete states without relying on hard boundary constraints, we employ an exact penalty scheme. We define a set of equality constraints that are satisfied only when the weights reach the discrete boundary:
\begin{equation}
    h_i(w) = w_i^2 - 1 = 0, \quad \forall i = 1, \dots, n.
\end{equation}
The penalty function is then constructed using the Euclidean norm of this constraint vector:
\begin{equation}
    M(w) = \norm{h(w)}_2 = \sqrt{\sum_{i=1}^n (w_i^2 - 1)^2}.
\end{equation}
This formulation smoothly penalizes any deviation from the valid discrete states. It creates an unconstrained exact penalty objective function:
\begin{equation}
    P_p(w) = f(w) + p M(w).
\end{equation}
Because the exact penalty natively enforces the structural bounds, we omit the projection operator entirely.

The step size follows a diminishing schedule $\gamma_k = 0.1 / k$. The penalty parameter $p$ is adaptively scaled by a factor of $\kappa = 1.1$ from an initial relaxed state of $p_0 = 10^{-3}$. As $p$ increases, the penalty term dominates, dynamically increasing the pressure on the continuous iterates to settle exactly into the discrete states.

\subsubsection*{Results and Evaluation}

As Table \ref{tab:binary_regression_results} shows, the algorithm achieves a perfect 100\% support recovery and low test MSE across all dataset sizes, finding the global minima of the problem every time. Crucially, the adaptive strategy automatically scaled the penalty parameter during execution, eliminating the need for manual tuning. This dynamic adjustment consistently yielded feasible, optimal solutions regardless of dataset scale, with final penalty values reported alongside the metrics.

\begin{table}[ht]
\centering
\caption{Binary regression performance. Columns: Number of observations, features, time, achieved objective, test MSE, final penalty parameter (P.P.), and support recovery.}
\label{tab:binary_regression_results}
\begin{tabular}{ccccccc}
\toprule
\textbf{N.Obs.} & \textbf{N.Fea.} & \textbf{Time (s)} & \textbf{Final Obj.} & \textbf{Test MSE} & \textbf{P.P.} & \textbf{Sup. Rec.} \\
\midrule
80    & 20    & 0.01    & 0.0112    & 0.00636    & 0.368 & 100\% \\
80    & 50    & 0.08    & 0.00787   & 0.0139     & 1.16  & 100\% \\
160   & 20    & 0.04    & 0.0103    & 0.00983    & 0.304 & 100\% \\
400   & 50    & 0.06    & 0.00992   & 0.0116     & 0.335 & 100\% \\
400   & 100   & 0.09    & 0.00947   & 0.0087     & 0.539 & 100\% \\
640   & 50    & 0.07    & 0.0105    & 0.009      & 0.368 & 100\% \\
800   & 200   & 0.24    & 0.0103    & 0.0106     & 0.539 & 100\% \\
800   & 500   & 1.58    & 0.0103    & 0.0104     & 1.05  & 100\% \\
1200  & 200   & 0.26    & 0.0103    & 0.0113     & 0.446 & 100\% \\
\bottomrule
\end{tabular}
\end{table}

\section{Summary}
The paper presents an approach to optimization problems with constraints in which objective functions depend on large number of variables and on distributions which can be estimated by samples of large dimensions. Due to objective functions characteristics these problems can be solved by transforming them to problems with penalty functions and possible with constraints of the type $x\in \X$ (which we call simple cosntraints). Since exact penalty functions, under some constraint qualifications, require finite values of penalty parameters in order to guarantee that solutions to problems with penalty functions are solutions to initial constrained optimization problems, it is natural to look for efficient methods for minimizing exact penalty functions under simple constraints. The paper presents constraint qualification which guarantees that solutions to problems with exact penalty functions and simple constraints are solutions to initial constrained optimization problems. 

Stochastic mirror descent methods are suitable methods for problems with simple constraints, furthermore they have proven convergence for problems defined by functions which belong to the class of functions which includes convex functions and certain non--convex functions (they define variationally coherent problems). So it is reasonable to adopt stochastic mirror functions to solving optimization problems with exact penalty functions and simple constraints. Since exact penalty functions are not differentiable the paper shows how to generalize the notion of variational coherence to problems with the functions admitting Clarke's generalized derivative. The paper indicates that due to semicontinuity of the generalized derivative, the proof of convergence of the mirror descent algorithm presented in \cite{boyd} holds in the case of problems with exact penalty functions.  

Numerical results of the application of the proposed method to some elementary constrained optimization problems are presented to empirically confirm the papers theoretical results. More numerical results concerning optimization problems related to learning quantized deep learning networks, and obtained by the method, will be shown elsewhere.


\bibliographystyle{unsrt}  
\bibliography{bibliography}


\newpage
\appendix
\label{secAppendix}

\section{Proof of the Lemma~\ref{l2}}\label{appendix:proofl2}
\begin{proof} (of {\it Lemma \ref{l2}}) Take any $x\in {\cal X}$.
	Let $r>0$ be a number such that
	\begin{eqnarray}
	&{\displaystyle
		M(x) < r}\nonumber
	\end{eqnarray}
	for all $x\in {\cal X}$. 
	We deduce from ${\bf (CQ)}$ that there is a simplex
	in $\cE(x)\subset
	\R^{n_E}$ with vertices $\{e_j\}_{j=0}^{n_E}$ which contains $0$
	as an interior point. By definition of $\cE(x)$, there exist
	$d_0,\ldots,d_{n_E}\in \D$ and $\delta > 0$ such that for $j=0,\ldots,n_E$
	\begin{eqnarray}
	&{\displaystyle 
		\left \{\left \langle \nabla h_i(x),
		d_j\right \rangle\right \}_{i\in E} =  e_j,\ \ 
		\max_{i\in I} \left \langle \nabla g_i(x),
		d_j \right \rangle  \leq  - \delta.}
	\nonumber 
	\end{eqnarray}
	
	Let $(\lambda_0,\lambda_1,\ldots,\lambda_{n_E})$ be the barycentric coordinates of $0$
	w.r.t. the vertices $e_j$ of the simplex, i.e.
	\begin{eqnarray}
	&{\displaystyle
		0\ \left (=\sum_{j=0}^{n_E} \lambda_j e_j\right ) = \nabla
		h (x)\circ \sum_{j=0}^{n_E} \lambda_j d_j .}\nonumber
	\end{eqnarray}
	Here
	\begin{equation*}
	\nabla h (x)\circ d := \left \{\left \langle \nabla h_i(x),d\right \rangle
	\right \}_{i\in
		E}.
	\end{equation*}
	We shall also write
	\begin{equation*}
	h (x) := \left \{ h_i(x)\right \}_{i\in E}.
	\end{equation*}
	
	Since the vertices are in general position and $0$ is an interior point, the
	$\lambda_i$'s are all positive and we may find $\delta_1 > 0$ such that
	\begin{eqnarray}
	&{\displaystyle
		\left (\lambda_0 - \sum_{j=1}^{n_E} \alpha_j,\lambda_1 + \alpha_1,\ldots,
		\lambda_{n_E} + \alpha_{n_E} \right )
		\in \left \{ \gamma\in \R^{n_E+1} \ \middle| \ \gamma_j \geq 0\ \forall
		j,\ \sum_{j=0}^{n_E}\gamma_j =1 \right \}} \nonumber
	\end{eqnarray}
	whenever $\alpha\in {\cal B} (0,\delta_1)\subset \R^{n_E}$. (${\cal B} (0,\delta_1)$
	is a ball with radius $\delta_1$.)
	Furthermore, the $n_E\times n_E$ matrix $P(x)$
	defined by
	\begin{eqnarray}
	&{\displaystyle
		P(x)\alpha := \sum_{j=1}^{n_E} \nabla h(x)\circ \alpha_j (d_j-d_0) }
	\label{b1}
	\end{eqnarray}
	is invertible for $x=\tilde{x}$, from the definition of $d_j,\ j= 1,\ldots,n_E$.
	
	In consequence of hypothesis ${\bf (CQ)}$ we may choose a neighborhood ${\cal N}(\tilde{x},\varepsilon)$ ($\varepsilon > 0$) of $\tilde{x}$ in $\X$ and
	numbers 	$\hat{r}\geq r$ and $\delta_2\in (0,\hat{r}^{-1}]$ such that for any
	$x\in \X$ satisfying $x\in {\cal N}(\tilde{x},\varepsilon)$
	\begin{eqnarray}
	(i) \hspace{5mm}& & \max_{i\in I}\left \langle \nabla g_i(x),v_j - x\right \rangle \leq
	- \delta/2\ \ \forall j   \nonumber\\
	(ii)\hspace{5mm} & & P(x)\ \ {\rm is\ invertible} \nonumber\\
	(iii)\hspace{5mm} & & \left \| P(x)^{-1} \nabla h(x)\circ
	\left (\left (\sum_{j=0}^{n_E} \lambda_j v_j\right )-x\right )
	\right \| \leq \delta_1/2\nonumber\\
	(iv)\hspace{5mm} & & \delta_2 \left \| P(x)^{-1} 
	\right \| n_E^{1/2} \leq \delta_1/2 .\nonumber
	\end{eqnarray}
	(In {\it (iv)} the norm is the Frobenius norm.
	Here the controls $v_j\in \X$, $j=0,\ldots,n_E$ are defined to be
	\begin{eqnarray}
	&{\displaystyle
		v_j:= x + d_j(x).}\nonumber
	\end{eqnarray}
	
	Now suppose that $x$ is not feasible.
	Set
	\begin{eqnarray}
	&{\displaystyle
		\alpha = P(x)^{-1} \left [ - \nabla h(x)\circ
		\left (\left (\sum_{j=0}^{n_E}\lambda_j v_j\right ) -x\right )
		- \delta_2 M(x)^{-1}
		h(x) \right ]. }
	\nonumber \\
	\label{b2}
	\end{eqnarray}
	Notice that, by properties {\it (iii)} and {\it (iv)},
	$\| \alpha \| \leq \delta_1$. Also set
	\begin{eqnarray}
	&{\displaystyle
		v = v_0 + \sum_{j=1}^{n_E} (\lambda_j + \alpha_j)(v_j-v_0).}\nonumber
	\end{eqnarray}
	Because $\| \alpha \| \leq \delta_1$ we have that $v\in \X$.
	  We now verify that $v$ has the required properties.                     
	Notice first that
	\begin{eqnarray}
	&{\displaystyle
		\left \| v - x \right \| \leq 2d,}\label{b3}
	\end{eqnarray}
	where $d$ is a bound on the norms of elements in $\X$.
	
	We have from (\ref{b1}) and (\ref{b2}) that
        \begin{align}
	P(x)\alpha & = \nabla h(x)\circ \sum_{j=1}^{n_E} 
	\alpha_j(v_j- v_0) 
	\nonumber \\
	& = - \nabla h(x)\circ \left (\sum_{j=1}^{n_E} \lambda_j \left (v_j-v_0\right ) 
	+ v_0 -
	x\right ) - \delta_2 M(x)^{-1}
	h(x).\nonumber
        \end{align}
	By definition of $v$,
	\begin{eqnarray}
	&{\displaystyle
		\nabla h(x)\circ (v-x) = -\delta_2 M(x)^{-1}h(x).}
	\end{eqnarray}
    But then
	\begin{eqnarray}
	&{\displaystyle
		\nabla h(x)\circ (v - x) = -(\delta_2/
		\hat{r})g(x).}
	\nonumber
	\end{eqnarray}
	Since $\delta_2/\hat{r}\leq 1$, it follows that
	\begin{eqnarray}
	&{\displaystyle
		\max_{i\in E} \left | h_i(x) + \left \langle \nabla h_i(x),
		v-x\right \rangle \right |
		- \max_{i\in E} \left | h_i(x)(x)\right |
		\leq - (\delta_2/\hat{r})M(x).}
	\end{eqnarray}
		
	We deduce from property {\it (i)} that
	\begin{eqnarray}
	&{\displaystyle
		\left \langle \nabla g_j(x),v_0 + \sum_{i=1}^{n_E}\left (\lambda_i +
		\alpha_i\right )\left (v_i-v_0\right )-x\right \rangle 
		\leq - \delta/2,\ \forall j\in I.}
	\nonumber
	\end{eqnarray}
	It follows that
	\begin{eqnarray}
	&{\displaystyle
		\left \langle \nabla g_j(x),v-x\right \rangle \leq - \delta/2,\ \forall j\in I.}\label{b5}
	\end{eqnarray}
%

Surveying inequalities (\ref{b3})--(\ref{b5}),
	we see that $v$ satisfies all relevant
	conditions for completion of the proof, 
	when we set $K_1 = \delta_2$, $K_2 =\delta/2$.


\end{proof}


\section{Derivation of the Proximal Variational Inequality}
\label{app:variational_inequality}

In Lemma \ref{lem:boundedPenalty_MD}, the finiteness proof relies on the variational inequality characterizing the trial step $X^+$. For completeness, we provide the derivation of this inequality directly from the properties of the mirror map and the convex conjugate.

Recall that during the penalty update loop, the trial step $X^+$ is evaluated from the reset dual state $Y = \der{h}{X_k}{}$. The updated dual variable is $Y^+ = \der{h}{X_k}{} - \gamma_k G_k$, and the trial primal point is generated via the mirror map as $X^+ = \Mirror_h(Y^+)$. 

By the definition of the Fenchel coupling framework, the mirror map $\Mirror_h(y) = \der{h^*}{y}{}$ yields the unique maximizer of the function defining the convex conjugate $h^*(y)$. Therefore, the trial point $X^+$ is the exact solution to the following maximization problem over the compact convex set $\X$:
\begin{equation}
    X^+ = \arg\max_{z \in \X} \left\{ \mul{Y^+}{z} - h(z) \right\}.
\end{equation}

Let the concave objective function of this subproblem be denoted as $\Psi(z) = \mul{Y^+}{z} - h(z)$. Because $X^+$ maximizes $\Psi(z)$ over the closed convex set $\X$, the first-order necessary and sufficient optimality condition dictates that the gradient of $\Psi$ evaluated at $X^+$ must satisfy the following variational inequality for all feasible directions $z \in \X$:
\begin{equation} \label{eq:first_order_opt}
    \mul{\der{\Psi}{X^+}{}}{z - X^+} \le 0.
\end{equation}

Taking the gradient of $\Psi(z)$ with respect to $z$ yields:
\begin{equation}
    \der{\Psi}{z}{} = Y^+ - \der{h}{z}{}.
\end{equation}
Evaluating this gradient at the maximizer $X^+$ and substituting it into the optimality condition (\ref{eq:first_order_opt}) gives:
\begin{equation}
    \mul{Y^+ - \der{h}{X^+}{}}{z - X^+} \le 0, \quad \forall z \in \X.
\end{equation}

Finally, we substitute the definition of the updated dual variable $Y^+ = \der{h}{X_k}{} - \gamma_k G_k$ into the inequality:
\begin{equation}
    \mul{\der{h}{X_k}{} - \gamma_k G_k - \der{h}{X^+}{}}{z - X^+} \le 0.
\end{equation}
Multiplying the entire inequality by $-1$ reverses the inequality sign and yields the required formulation used in Lemma \ref{lem:boundedPenalty_MD}:
\begin{equation}
    \mul{\gamma_k G_k + \der{h}{X^+}{} - \der{h}{X_k}{}}{z - X^+} \ge 0, \quad \forall z \in \X.
\end{equation}

\end{document}